\documentclass[a4paper,10pt,twoside]{article}
\usepackage{amsmath,amsthm,amsfonts}
\usepackage[pdfborder={0 0 0},pagebackref=false]{hyperref}
\usepackage{geometry}
\usepackage{color}
\title{Ordered and size-biased frequencies\\ in GEM and Gibbs models for species sampling}
\author{Jim Pitman\footnote{Statistics Department, 367 Evans Hall \# 3860, University of California, Berkeley, CA 94720-3860, U.S.A.
    \texttt{pitman@berkeley.edu}}
  \and %% remove this line and below if single author
 Yuri Yakubovich\footnote{Saint Petersburg State University, St.\;Petersburg State University, 7/9 Universitetskaya nab., St.\;Petersburg, 199034 Russia. \texttt{y.yakubovich@spbu.ru}}}

\newcommand{\perms}[1]{\mathcal{S}_{#1}}

\newcommand{\GEM}{\mathsf{GEM}}
\newcommand{\RAM}{\mathsf{RAM}}

\newcommand{\EPPF}{\mathsf{EPPF}}
\newcommand{\OEPPF}{\mathsf{OEPPF}}
\newcommand{\CRP}{\mathsf{CRP}}
\newcommand{\OCRP}{\mathsf{OCRP}}
\newcommand{\ISBP}{\mathsf{ISBP}}

\newcommand{\BN}{\mathbb{N}}

\newcommand{\CC}{ {\mathcal C }}

\newcommand{\BP}{\mathbb{P}}

\newcommand{\BE}{\mathbb{E}}

\newcommand{\Pbul}{P_{\bullet}}

\newcommand{\ed}{\overset{d}{{}={}}}

\newcommand{\tilPi}{\widetilde{\Pi}}
\newcommand{\tilC}{\widetilde{C}}

\newcommand{\sizem}{\mbox{(size$-\alpha$)}}

\newcommand{\giv}{\,|\,}  %% given,  for conditioning

\newcommand{\Nright}{N^{X\uparrow} }
\newcommand{\NYup}{N^{Y\uparrow} }

\newcommand{\Ndec}{N^{\downarrow}}
\newcommand{\Ns}{N}
\newcommand{\psize}{   \widetilde{s}    }
\newcommand{\pcomp}{\widehat{p}}
\newcommand{\tilp}{\widetilde{p}}
\newcommand{\sigmapi}{\pi}

\newcommand{\Nstar}{N^{*}}

\newcommand{\Nsum}{\nu}

\definecolor{gray}{gray}{0.4}
\newtheorem{theorem}{Theorem}

\newtheorem{lemma}[theorem]{Lemma}
\newtheorem{corollary}[theorem]{Corollary}
\theoremstyle{definition}
\newtheorem*{remark*}{Remark}

\numberwithin{equation}{section}
\numberwithin{theorem}{section}
\numberwithin{remark}{section}

\begin{document}

\maketitle

\begin{abstract}
We describe the distribution of frequencies ordered by sample values in a random sample of size $n$ from the 
two parameter $\GEM(\alpha,\theta)$ random discrete distribution on the positive integers. These frequencies 
are a $\sizem$-biased random permutation of the sample frequencies in either ranked order, or in the order of appearance of values in the
sampling process. This generalizes a well known identity in distribution due to Donnelly and Tavar\'e (1986) for $\alpha = 0$
to the case $0 \le \alpha < 1$.
This description extends to sampling from Gibbs$(\alpha)$ frequencies obtained by suitable conditioning of the
$\GEM(\alpha,\theta)$ model, and yields a value-ordered version of the Chinese Restaurant construction of
$\GEM(\alpha,\theta)$ and Gibbs$(\alpha)$ frequencies 
in the more usual size-biased order of their appearance.
The proofs are based on a general construction of a finite sample $(X_1,\dots,X_n)$ from any random frequencies in size-biased
order from the associated exchangeable random partition $\Pi_\infty$ of $\BN$ which they generate.
\end{abstract}

\tableofcontents

%MSC 60C05  60G09  

%Keywords: Species sampling; random exchangeable partition; size-biased order; GEM distribution; Gibbs partitions; Chinese Restaurant construction

\newpage

\section{Introduction}

We are interested in various natural orderings of clusters of common values, in sampling from random discrete distributions.
The recent review by Crane \cite{MR3458585} \cite{MR3458591} %%%  AUTHOR = {Crane, Harry},    TITLE = {The ubiquitous {E}wens sampling formula},
%\cite{MR3335108} %%%, AUTHOR = {Crane, Harry}, TITLE = {Generalized {E}wens-{P}itman model for {B}ayesian clustering},
presents some of the widespread applications of these models of random partitions,  in the contexts of population genetics, 
ecology, Bayesian nonparametrics, combinatorial stochastic processes, and inductive inference. 
Let $\Pbul:= (P_1, P_2, \ldots)$ denote a random discrete distribution on the positive integers, to be thought of as a model
for population frequencies of various species in a large population of individuals classified by species, or otherwise partitioned by
type in some way. We take these population frequencies to be
represented in the {\em stick-breaking form}
\cite{MR0010342}, %%, AUTHOR = {Halmos, Paul R.}, TITLE = {Random alms},
\cite{sawyer1985sampling} %%, title={A sampling theory for local selection}, author={Sawyer, Stanley and Hartl, Daniel},
\begin{equation}
\label{stickbreak}
P_1:= H_1, \qquad P_2 = (1-H_1) H_2, \qquad P_3 := (1-H_1) (1-H_2) H_3, \qquad \ldots
\end{equation}
for some sequence of random variables $H_i$ with values in $(0,1)$ such that $P_i >0$ for all $i$ and $\sum_{i=1}^\infty P_i = 1$ almost surely.
We call this model for population frequencies a {\em residual allocation model} ($\RAM$) to indicate the $H_i$ are independent, 
though not necessarily identically distributed.
Let $X_1, \ldots, X_n$ denote a {\em sample of size $n$} from population frequencies $\Pbul$,  that is the first $n$ terms of a sequence $X_1, X_2, \ldots$ which 
conditionally given $\Pbul$ is i.i.d.\ according to $\Pbul$.  
We are most interested in settings where the possible sample values $1,2, \ldots$ have a clear meaning in the context of some larger population model, such 
as the age-ranks of alleles in the infinitely-many-neutral-alleles diffusion model 
\cite{MR1066961}. %%%, AUTHOR = {Ethier, S. N.}, TITLE = {The distribution of the frequencies of age-ordered alleles in a diffusion model},
This model involves $\Pbul$ with $\GEM(0,\theta)$ distribution, while other  models of current interest 
\cite{MR2596654} %%, AUTHOR = {Petrov, L. A.}, TITLE = {A two-parameter family of infinite-dimensional diffusions on the {K}ingman simplex},
\cite{costantini2016wright} %%, title={Wright-Fisher construction of the two-parameter Poisson-Dirichlet diffusion},
involve ranked frequencies derived from the $\GEM(\alpha,\theta)$ distribution,
in which $H_i$ has the beta$(1-\alpha, \theta + i \alpha)$ density on $(0,1)$. Here $0 \le \alpha < 1$ and $\theta>-\alpha$ are real parameters, 
and beta$(a,b)$ is the probability distribution on $(0,1)$ with probability density proportional to $u^{a-1}(1-u)^{b-1}$ at $0 < u < 1$.
See 
\cite{MR2245368},  %%% CSP 
\cite{MR2663265} %%%, AUTHOR = {Feng, Shui}, TITLE = {The {P}oisson-{D}irichlet distribution and related topics},
for further background on $\GEM$ distributions.
As discussed further in following sections, a special property of $\GEM(\alpha,\theta)$, important in many contexts,
is that the $\GEM(\alpha,\theta)$ frequencies $\Pbul$ are in a size-biased random order.
It is also known 
\cite{MR1387889}  %%, AUTHOR = {Pitman, Jim}, TITLE = {Random discrete distributions invariant under size-biased permutation},
\cite{MR2744243} %%, AUTHOR = {Gnedin, Alexander and Haulk, Chris and Pitman, Jim}, TITLE = {Characterizations of exchangeable partitions and random discrete distributions by deletion properties}.
that
\begin{equation}
\label{gemisbp}
\mbox{ the only $\RAM$s with frequencies in size-biased order are the $\GEM(\alpha,\theta)$ models.}
\end{equation}
Combined with the fact that the large $n$ asymptotics of $\GEM(\alpha,\theta)$ samples exhibit a variety of logarithmic and power law behaviors as $(\alpha,\theta)$ 
varies 
\cite{MR2245368}  %%% CSP 
this draws attention to the $\GEM(\alpha,\theta)$ family as a tractable and versatile family of models for use in applications.

A basic problem 
is  to describe
the distribution of the partition of $n$ determined by the {\em size-ordered} or {\em ranked sample frequencies}, 
meaning the list of sizes of clusters of equal values in a sample from a random discrete distribution $\Pbul$,
in decreasing  order of size.
As recalled in Section \ref{sec:bac}, that problem has been solved for both $\GEM(\alpha,\theta)$ and for $\RAM$s with i.i.d.\ factors, 
hand in hand with a description of the distribution of the {\em appearance-ordered}  sample frequencies, that
is the list of sizes of clusters of equal values in the order in which these values appeared in the random sampling process.
It is well known that in a sample from any random discrete distribution $\Pbul$,
the appearance-ordered sample frequencies are a size-biased permutation of the partition of $n$.
A more difficult problem is to provide a corresponding description of the 
{\em value-ordered sample frequencies} which can be obtained by ordering the sample in (weakly) increasing order and then reading counts of equal values, so that the number of times the minimal value in the sample is attained comes first, 
and the frequency of the maximal sample value comes last.  See a discussion and an example in Section \ref{sec:partitionsfromsamples} below that should clarify these notions.

For the $\GEM(0,\theta)$, with i.i.d.\ beta$(1,\theta)$ factors $H_i$, 
a remarkably simple description of the value-ordered sample frequencies was provided by
Donnelly and Tavar\'e \cite{MR827330}:  %%AUTHOR = {Donnelly, Peter and Tavar{\'e}, Simon}, TITLE = {The ages of alleles and a coalescent}, e(1986 age-ordering of alleles) 
\begin{quote}
 in sampling from $\GEM(0,\theta)$,
there is no difference in distribution between  the value-ordered frequencies and the appearance-ordered frequencies:
 they are both in a size-biased random order.
\end{quote}
For sampling from a $\RAM$ with i.i.d.\ factors, a more complicated description of the distribution of value-ordered frequencies was found in 
Gnedin and Pitman \cite[\S 11]{MR2122798}.  %%%AUTHOR = {Gnedin, Alexander and Pitman, Jim}, TITLE = {Regenerative composition structures}.
But there seems to be a gap in the literature regarding the distribution of value-ordered frequencies for a $\GEM(\alpha,\theta)$ sample for $0 < \alpha < 1$.
%This problem, which involves independent but not identically distributed stick-breaking factors $H_i$, seems very difficult to solve by a direct approach using the $\RAM$.
Our main result is that this problem has %nonetheless 
a surprisingly simple solution, almost as simple as the Donnelly and Tavar\'e result for $\alpha = 0$.
Compared to the case $\alpha = 0$, the only difference is that the usual notion of 
size-biased permutation of a composition $(n_1, \ldots, n_k)$ of $n$ needs to be
replaced by the notion of a {\em $\sizem$-biased random permutation}, defined as follows: 
\begin{itemize}
\item for each $1 \le i \le k$ the first component is set equal to $n_i$ with probability $\frac { n_i-\alpha}{n - k \alpha}$, 
\item  given $k > 1$ and that the $i$th component was chosen to be the first, for each $j \ne i$ the second component is set equal to $n_j$ with probability $\frac{n_j-\alpha}{n - n_i - (k-1) \alpha}$ 
\end{itemize}
and so on, as discussed more carefully in Section \ref{sec:pseudointro} and also in Appendix \ref{sec:pseudo}.

\begin{theorem}
\label{thm:sbm}
For each $n \ge 1$, in a random sample of size $n$ from $\GEM(\alpha,\theta)$, the value-ordered sample frequencies are $\sizem$-biased.
\end{theorem}

Because sample frequencies in appearance order are size-biased in the usual way, this theorem shows that the Donnelly-Tavar\'e identity in distribution
between value-ordered and appearance-ordered frequencies is very special to $\GEM(0,\theta)$. 
It does not extend to $\GEM(\alpha,\theta)$ for $0< \alpha < 1$ without extending the notion of size-biasing to $\sizem$-biasing.
%As it is known that $\GEM(\alpha,\theta)$ frequencies are in a size-biased order,
Hence the theorem dispels the tempting but false idea that value-ordered and appearance-ordered sample frequencies might be identically distributed in any
model with value-ordered population frequencies in size-biased order.  For except in trivial cases of equality between counts,
for $0 < \alpha < 1$ a $\sizem$-biased permutation is not the same in distribution as simple size-biased permutation.

Our proof of Theorem \ref{thm:sbm} in Section \ref{sec:main} shows much more: according to Theorem \ref{thm:gibbs},
the conclusion of Theorem \ref{thm:sbm} holds also for sampling from the size-biased presentation of 
frequencies  of any {\em Gibbs$(\alpha)$ partition}. That is for $\Pbul$ derived as the limits of relative frequencies in order of appearance of any random partition $(\Pi_n)$ of positive integers 
with the conditional distribution of $\Pi_n$ given $K_n=k$ that is shared by all $\GEM(\alpha, \theta)$ partitions 
\cite[Theorem 4.6]{MR2245368}  %%% CSP 
\cite{MR2160320}, %%%AUTHOR = {Gnedin, A. and Pitman, J.}, TITLE = {Exchangeable {G}ibbs partitions and {S}tirling triangles},
described in more detail in Section~\ref{sec:gibbs} below.
This leads us to speculate that there is a converse to Theorem \ref{thm:gibbs}: if in sampling from $\Pbul$ the value-ordered clusters are $\sizem$-biased, 
then $\Pbul$ is the size-biased presentation of some Gibbs$(\alpha)$ frequencies. But we do not have any proof of this.

The case $\alpha = 0$ of Theorem \ref{thm:sbm},
due to Donnelly and Tavar\'e \cite{MR827330},  %%AUTHOR = {Donnelly, Peter and Tavar{\'e}, Simon}, TITLE = {The ages of alleles and a coalescent}, e(1986 age-ordering of alleles) 
was a culmination of earlier work  by Watterson and others
\cite{MR0475994} %%, AUTHOR = {Watterson, G. A.}, TITLE = {Reversibility and the age of an allele. {I}. {M}oran's infinitely many neutral alleles model}, JOURNAL = {Theoret. Population Biology}, FJOURNAL = {Theoretical Population Biology. An International Journal},
\cite{MR0475995}    %, AUTHOR = {Watterson, G. A.}, TITLE = {Reversibility and the age of an allele. {II}. {T}wo-allele models, with selection and mutation},
\cite{watterson1977most}  %%  , title={Is the most frequent allele the oldest?}, author={Watterson, GA and Guess, HA},
on random sampling from models of limit populations in genetics with random frequencies
governed by $\GEM(0,\theta)$ when listed in order of {\em age-rank},
meaning the frequencies of the oldest, second-oldest, third-oldest, 
$\cdots$
alleles in the population. 
The age-ranked frequencies in these models are in size-biased random order, and 
the $i$th sample value $X_i$ in this setting is then the age rank in the large population of the allelic types of the $i$th individual to be sampled.
Thus it reasonable to study samples from more general frequencies in size-biased order thinking of the sample value as of the age-rank in an infinite idealized 
population.
A natural question in this setting, is given that the allelic composition of a sample is $(n_1, \ldots, n_k)$, what is the
probability that a particular allele with $n_i$ representatives is the oldest in the sample? According a result of
Watterson and Guess \cite[Theorem 3]{watterson1977most}, under assumptions that 
are known \cite{MR1104078}  %% AUTHOR = {Donnelly, Peter and Joyce, Paul}, TITLE = {Consistent ordered sampling distributions: characterization and convergence},
to imply $\GEM(0,\theta)$ frequencies by age-rank in the limit population, the allele with $n_i$ representatives is the oldest in the sample with probability $n_i/n$.
Theorem \ref{thm:gibbs} extends this result as follows:

\begin{corollary}
In sampling from  a limit population with frequencies by age-rank which are in a size-biased random order, and distributed
according either to $\GEM(\alpha,\theta)$,  or to the size-biased presentation of frequencies in a Gibbs$(\alpha)$ model,
the allele with $n_i$ representatives in the sample composition $(n_1, \ldots , n_k)$ is the oldest in the sample with probability $(n_i-\alpha)/(n - k \alpha)$.
\end{corollary}

There is a combinatorial construction of the Gibbs$(\alpha)$ sample frequencies in size-biased appearance order
known as the Chinese Restaurant Process ($\CRP$)  \cite[\S 3.1]{MR2245368}.  %%% CSP 
Our proof of Theorem \ref{thm:gibbs} and its corollary Theorem \ref{thm:sbm} 
involves 
Ordered Chinese Restaurant Process ($\OCRP$),
which produces value-ordered sample frequencies,
as considered in 
\cite[\S 5.2.1]{MR2275248}  %%, AUTHOR = {James, Lancelot F.}, TITLE = {Poisson calculus for spatial neutral to the right processes},
and \cite[\S 11]{MR2122798}, %%AUTHOR = {Gnedin, Alexander and Pitman, Jim}, TITLE = {Regenerative composition structures},
 and discussed further in Section \ref{sec:main}.
In comparing the  prescriptions for appearance-ordered and value-ordered frequencies in sampling from a Gibbs$(\alpha)$ model,
while there are obvious similarities between the two schemes, there are also subtle differences.
Observe first that if you are given {\em both} the appearance-ordered and the value-ordered frequencies, by exchangeability the appearance-ordering is just a size-biased
random ordering of the value-ordered frequencies.  
So the value-ordered frequencies serve as a sufficient statistic for predictions about the next sample value $X_{n+1}$.
It is a subtle point of the prediction rule given the value-ordered frequencies, spelled out in 
Corollary \ref{crl:cndind}
below, that the value-ordered frequencies provide no more information than the appearance-ordered frequencies,
so far as predicting whether $X_{n+1}$ will be a new value or not: all that matters is the number of existing clusters $k$ and the sample size $n$: the probability that the next observation is a new value depends only on $n$ and $k$,
no matter what the appearance-ordered or value-ordered frequencies of the $k$ clusters. This is a very special property of Gibbs$(\alpha)$ models, closely
associated with the Markov property of $K_n$ for these models.
What is even more interesting, considering that the probability of a new value is unaffected by the value-ordered frequencies, is that 
the value-ordered frequencies {\em do} affect the probabilities that the new observation equals one of the clusters of previous observations, as is plain from comparison of the two formulas 
\eqref{app:sb} and \eqref{freq1:sb} below. The sequential scheme for selecting a value is the same in both cases, except that 
the scheme given value-ordered frequencies puts weight $n_1 + 1 - \alpha$ on the lowest-valued cluster and 
weight $n - n_1 - (k-1) \alpha$ on the rest, whereas the scheme given appearance-ordered frequencies puts lesser weight $n_1 - \alpha$ on the first cluster to appear, 
and the same weight 
$n - n_1 - (k-1) \alpha$ 
on the rest.
So the value-ordered frequency data changes the prediction of the next observation given it is one of those previously observed,  
always pushing it to be more likely to be the lowest previous value observed $n_1$ times,  no matter what the previously observed frequencies in value-order $n_1, \ldots , n_k$.

\begin{corollary}
\label{crl:cndind}
In sampling from the limit frequencies of 
any Gibbs$(\alpha)$ model in size-biased order, 
conditionally given the number $K_n$ of distinct values in the sample, the event $X_{n+1} \notin \{X_1, \ldots, X_n\}$ of a new value at time $n+1$ is independent of the value-ordered frequencies of the sample $X_1, \ldots, X_n$.
In other words, the conditional probability of this event given value-ordered frequencies $(n_1, \ldots, n_k)$ in a sample with $K_n = k$
depends only on $n$ and $k$
and does not depend otherwise on $(n_1, \ldots, n_k)$.
\end{corollary}

The rest of this article is organized as follows. In the next Section we introduce the notation and recall some notions we use.
In Section \ref{sec:main} we formulate and prove our main result, Theorem \ref{thm:gibbs}, and also discuss the $\OCRP$ which 
produces value-ordered sample frequencies step by step. 
In Section \ref{sec:related} we 
present an alternative computational proof of Theorem \ref{thm:sbm} and also derive some consequences from the $\OCRP$
description of Corollary \ref{crl:ocrp}. This allows us to reproduce well-known results for the $\GEM(0,\theta)$ model
with this new approach, thus providing an additional check for it. Finally, in appendices we collect 
some basic facts about a generalization of the $\sizem$-biasing procedure, and compare the value-ordering used in this paper with the regenerative ordering 
of \cite{MR2122798}. %%%%%, AUTHOR = {Gnedin, Alexander and Pitman, Jim}, TITLE = {Regenerative composition structures},. 

\section{Background and notation}
\label{sec:bac}

\subsection{Partitions generated by random samples}
\label{sec:partitionsfromsamples}

Let $\Pi_n$ denote the random partition of $n$ generated by the sample values $X_1, \ldots, X_n$. That is, 
if there are say $K_n= k$ distinct sample values $X_1, \ldots, X_n$, 
the partition  of the set $[n]:= \{1, \ldots, n\}$ is
\begin{equation}
\label{pinc}
\Pi_n := \{ C_1, \ldots, C_k \}
\end{equation}
with $C_1:= \{i \le n: X_i = X_{M(1)} \}$,
where $M(1) = 1$ is the least element of both $[n]$ and $C_1$,
and if $k \ge 2$ then $C_2:= \{i \le n: X_i = X_{M(2)} \}$, where $M(2)$ is the least element of both $[n] \setminus C_1$ and of $C_2$,
and so on. These {\em clusters $C_i$} generated by the sample, are listed here in their {\em order of appearance}.  
We are interested in various orderings of these clusters.
Each ordering of clusters induces a list of their sizes in that order, 
which is a random sequence of strictly positive integers with sum $n$,
called a {\em random composition of $n$}. 
The  {\em value} of cluster $C_i$ is the common value of $X_j$ for every $j \in C_i$.
We are particularly concerned with:
\begin{itemize}
\item The {\em cluster sizes in appearance-order}: $\Nstar_{\bullet:n}:= (n_1, \ldots, n_k)$ if cluster $C_j$ as above has $n_j$ members for each $1 \le j \le k$.
\item The {\em clusters in value-order} define a {\em random ordered partition of $[n]$} 
\cite[\S 11]{MR2122798} %%AUTHOR = {Gnedin, Alexander and Pitman, Jim}, TITLE = {Regenerative composition structures},
\cite[\S 5.2.1]{MR2275248}  %%, AUTHOR = {James, Lancelot F.}, TITLE = {Poisson calculus for spatial neutral to the right processes},
\begin{equation}
\label{tilpinc}
\tilPi_n:= (\tilC_1, \ldots , \tilC_k):= (C_{\sigmapi(1:n)}, C_{\sigmapi(2:n)}, \ldots , C_{\sigmapi(k:n)})
\end{equation}
for some permutation $\sigmapi(\bullet:n)$ of $[k]$, which encodes the additional value-order structure.  
Explicitly:
\begin{equation}
\label{tilpinc1}
\tilC_1:= \{ i \in [n] : X_i = \smash[b]{\min_{j \in [n] } X_j \} }
\end{equation}
and if $\tilC_1 \ne [n]$ then
\begin{equation}
\label{tilpinc2}
\tilC_2:= \{ i \in [n] : i \notin \tilC_1, X_i = \min_{j \in [n] \setminus \tilC_1} X_j \} 
\end{equation}
and so on.
Notice that for $\tilPi_n$ with $K_n$ clusters the permutation $\sigmapi(\bullet:n)$ of $[K_n]$ is encoded in the state $\tilPi_n$: the inverse of $\sigmapi(\bullet:n)$ 
is obtained by rearranging the clusters  $(\tilC_1, \ldots , \tilC_k)$ in order of their least elements.
\item the {\em cluster sizes in value-order\/}: $\Nright_{\bullet:n}$ with $\Nright_{i:n}  = \# \tilC_i = n_{\sigmapi(i:n)}$ for $1 \le i \le K_n$.
This is the sequence of sizes of clusters in increasing  order of their common $X$-values. For instance 
\begin{gather*}
\mbox{ $\Nright_{1:n}$ is the number of $j$ such that $X_j = \min_{1 \le i \le n} X_i$,  }\\
\mbox{ $\Nright_{K_n:n}$ is the number of $j$ such that $X_j = \max_{1 \le i \le n} X_i$. }
\end{gather*}
\item The {\em partition of $n$ generated by the sample}: $\Ndec_{\bullet:n}$  is the weakly decreasing rearrangement of either $\Nstar_{\bullet:n}$ or $\Nright_{\bullet:n}$.
\end{itemize}

We illustrate these definitions by an adaptation of Kingman's paintbox model 
\cite{MR509954} %%, AUTHOR = {Kingman, J. F. C.}, TITLE = {The representation of partition structures},
 for generating random partitions.
Let $(I_j)$ be the interval partition of $[0,1]$ defined by
\begin{equation}
\label{sbintpart}
I_1 = (0, P_1), \qquad I_2 = (P_1, P_1 + P_2), \qquad   I_3 = (P_1 + P_2, P_1 + P_2 + P_3) 
\end{equation}
and so on. Define the sample values by $X_i = j$ if $U_i \in I_j$ where $U_i$ is a sequence of i.i.d.\ uniform$[0,1]$ variables. 
In the following display a particular realization of the successive partial sums of probabilities $P_1, P_1 + P_2, \ldots$ is marked by a series of vertical bars $|$ in
a unit interval $[0,1]$. Then $n = 6$ sample points $U_i$ picked from $[0,1]$ landed between the bars as indicated:
\begin{equation}
\label{exampleU}
[\underbrace{0   \vphantom{U_X}      \qquad \quad}_1 |\underbrace{    \vphantom{U_X}\quad     }_2  | \underbrace{  \quad  U_6 \;  \quad  U_5  \,}_3| 
\underbrace{ \vphantom{U_X} \qquad}_4 |\underbrace{ U_3 }_5 |\underbrace{\;\vphantom{U_X} }_6  |\underbrace{\quad\vphantom{U_X} }_7  | \underbrace{\; U_2 \quad U_1 \qquad U_4}_8|\underbrace{ \quad\vphantom{U_X} }_{9\lefteqn{\qquad\ldots}}  | %\qquad | \quad 
|||||..... 1]
\end{equation}
Regarding the bars as separators between interval boxes with labels $1,2,3, \ldots$ shown under the braces, various quantities under consideration are in this instance:
\begin{itemize}
\item the sample from $\Pbul$ is $(X_1, \ldots, X_6) = (8,8,5,8, 3,3)$;
\item the partition of $[6]$ is $\Pi_6 = \bigl\{ \{1,2,4 \}, \{3\}, \{5,6\} \bigr\}$;
\item the clusters in appearance-order are $(C_1, C_2, C_3) = \bigl( \{1,2,4 \}, \{3\}, \{5,6\}\bigr)$;
\item the cluster sizes in appearance-order are $\Nstar_{\bullet:6} = (3,1,2)$;
\item the clusters in value-order are $(\tilC_1,\tilC_2,\tilC_3)=(C_3, C_2, C_1) = \bigl( \{5,6\}, \{3\}, \{1,2,4 \}\bigr)$;
\item the cluster sizes in value-order are $\Nright_{\bullet:6} = (2,1,3)$, corresponding to numbers of repeated values in the increasing rearrangement $(3,3,5,8,8,8)$ of the sample; 
\item the partition of $6$ defined by the cluster sizes is  $\Ndec_{\bullet:6} = (3,2,1)$.
\end{itemize}

%\noindent
For any partition $\{C_1,\dots,C_k\}$ of $[n]$ the probability 
of the basic event \eqref{pinc} that this particular partition is generated
by an exchangeable sample $X_1,\dots,X_n$
 depends just on cluster sizes $n_i=\#C_i$, so defines a
function $p(n_1,\dots,n_k)$  of compositions 
$(n_1,\dots,n_k)$  of $n$.
Following \cite[\S 2.2]{MR2245368},  %%% CSP 
this function of compositions of $n$ is called  the {\em exchangeable partition probability function} $(\EPPF)$ of $\Pi_n$.
For each fixed $k$ the $\EPPF$ is a symmetric function of $k$ positive integer arguments. As $n$ varies, the $\EPPF$ satisfies 
an addition rule
\cite[(2.9)]{MR2245368} %%%CSP
reflecting the consistency property of the random partitions, that $\Pi_m$ is the restriction to $[m]$ of $\Pi_n$ for each $m < n$.
However one can also consider the
$\EPPF$ for a fixed $n$, as we do in Lemmas \ref{lmm:sbias} and \ref{lmm:oeppf} in Appendix \ref{sec:pseudo}. 

Similarly, for an ordered partition $(C_1,\dots,C_k)$ of $[n]$ in some order,
with $n_i=\#C_i$, the probability $\tilp(n_1,\dots,n_k)$ of the event 
that this particular ordered partition is obtained by some ordered partition construction from an exchangeable sample is called an
{\em ordered exchangeable partition probability function} ($\OEPPF$).
This function may no longer be symmetric in $(n_1, \ldots, n_k)$. The term exchangeable
means only that the probability of achieving the ordered partition $(C_1,\dots,C_k)$ 
depends only on sizes $(n_1,\dots,n_k)$ of clusters of the partition.
 As $n$ varies, the $\OEPPF$ will also satisfy some consistency relations, 
see \cite[Eqs. (2), (3)]{MR2122798}.

It is a well known consequence of exchangeability of $\Pi_n$, that no matter what $\Pbul$
\begin{equation}
\label{starsb}
\mbox{ $\Nstar_{\bullet:n}$ is a size-biased random permutation of $\Ndec_{\bullet:n}$, as well as of $\Nright_{\bullet:n}$.}
\end{equation}
%The definition of a size-biased random permutation is recalled in Section \ref{sec:pseudo}.
As the sample size $n \to \infty$, it follows easily from the strong law of large numbers and \eqref{starsb} that no matter what the distribution of $\Pbul$,
there is the almost sure convergence of relative cluster sizes
\begin{equation}
\label{asconv}
n^{-1}\,\Nright_{\bullet:n} \to \Pbul \mbox{ and } n^{-1}\,\Nstar_{\bullet:n} \to \Pbul ^* \mbox{ almost surely }
\end{equation}
where $\Pbul^*$ is  a size-biased random permutation of $\Pbul$.  Consequently
\begin{equation}
\label{isbp}
\mbox{ if $\Pbul$ is in a size-biased random order, then $\Pbul \ed \Pbul^*$. }
\end{equation}
Such a random discrete distribution $\Pbul$ is said to be {\em invariant under size-biased permutation } ($\ISBP$).
This condition plays a central role 
in the theory of partitions generated by random sampling, for a number of reasons. One reason is that if $\Pbul$ is $\ISBP$, there is a 
simple general formula 
\cite{MR1337249}
\cite[(3.4)]{MR2245368}  %%% CSP 
for the probability of the basic event \eqref{pinc} for any particular partition $\{ C_1, \ldots, C_k \}$ of $[n]$, in terms of multivariate  moments of $\Pbul$. Another reason 
is that distributions that are $\ISBP$
are precisely the distributions of frequencies of species in appearance-order in any exchangeable species sampling model with proper frequencies  
\cite{MR2245368}.  %%% CSP  xxx better reference?
Especially in contexts where there 
is no a priori natural ordering of frequencies by positive integers, for instance in the setting of population genetics where different alleles might be identified only by some biochemical tag,
or in the theory of interval partitions generated by the zeros of stochastic processes 
\cite{MR2245368}  %%% CSP  xxx better reference?
one may as well use the size-biased ordering whenever that is tractable, because of its close connection to partition probabilities.

\subsection{Pseudo-size biased random order}
\label{sec:pseudointro}

As it was mentioned in the Introduction, the value-ordering of clusters of $\Pi_n$ involves the procedure of $\sizem$-biased random
permutation of component sizes.  We treat the $\sizem$-biased permutation an instance of a more general 
notion of an $s$-biased permutation \cite{MR1156448}, where $s$ is some strictly positive function of a cluster size called {\em pseudo-size}.  
For a partition $\Pi_n=\{C_1,\dots,C_k\}$ of $[n]$, with $\#C_i=n_i$  for $i\in[k]$, and such a function $s$,
an {\em $s$-biased pick} is a randomly chosen cluster of $\Pi_n$, which given $\Pi_n$ equals cluster $C_i$ with
probability $s(n_i)\big/(s(n_1)+\dots+s(n_k))$. An {\em $s$-biased permutation} is a random permutation 
of clusters in order of an exhaustive process of sampling without replacement by a sequence of $s$-biased picks. 
The usual {\em size-biased permutation} of $\Pi_n$ is just its $s$-biased permutation for 
the choice of the pseudo-size $s(m):=m$ to be just the ordinary size. 

With the pseudo-size function $s$ we associate the probability 
\begin{equation}
\label{psize:def}
\psize( n_1, \ldots, n_k):= \prod_{i=1}^k \frac{ s(n_i) }  { s(n_i) + \cdots + s(n_k) }  
\end{equation}
that an ordered collection of clusters $(C_1,\dots,C_k)$ of sizes $(n_1,\dots,n_k)$ stays exactly in the same order
after $s$-biased permutation. Notice that this is {\em not} the probability that after $s$-biased
permutation of clusters their sizes will be $(n_1,\dots,n_k)$. That probability is $\psize( n_1, \ldots, n_k)$ multiplied by an appropriate combinatorial coefficient.
We shall need the following result which we prove in  Appendix \ref{sec:pseudo}.

\begin{lemma}
\label{lmm:order}
Let\/ $\Pi_n$ be an exchangeable random partition of\/ $[n]$ with\/ $\EPPF$ $p$. Consider a strictly positive size function $s(m)$ 
of positive integers $m \le n$, and associate with it a function $\psize$ of compositions of $n$ as in \eqref{psize:def}.
\begin{itemize}
\item[(i)] If\/ $\Ns_{\bullet:n}$ is a listing of 
sizes of clusters of an $s$-biased permutation of\/ $\Pi_n$, then the probability function of\/ $\Ns_{\bullet:n}$ on compositions of\/ $n$ is
\begin{equation}
\label{s-bias}
\BP[ \Ns_{\bullet:n} = (n_1, \ldots , n_k) ] = {n \choose n_1, \ldots , n_k} \, \psize (n_1, \ldots, n_k) \, p(n_1, \ldots, n_k). 
\end{equation}
%where $\psize (n_1, \ldots, n_k):=  \prod_{i=1}^k s(n_i) (s(n_i) + \cdots + s(n_k) )^{-1}$ as in \eqref{psize:def}.
\item[(ii)]
Conversely, if the probability function of a random composition $\Ns_{\bullet:n}$ is given by \eqref{s-bias} for all compositions of $n$, for
some symmetric function $p$ of compositions, then  $p$ is an $\EPPF$, and 
$\Ns_{\bullet:n}$ has the same distribution as an $s$-biased random ordering of component sizes of $\Pi_n$ with $\EPPF$ $p$.
\end{itemize}
\end{lemma}

\subsection{Gibbs partitions}
\label{sec:gibbs}

We are particularly interested in the $\EPPF$s which can be represented in the Gibbs form
\begin{equation}
\label{gibbseppf}
p(n_1,\dots,n_k)=V_{k:n}\prod_{i=1}^kw(n_i),\qquad\mbox{where }n=\sum_{i=1}^kn_i,
\end{equation}
for some positive weights $V_{k:n}$, $1 \le k \le n$ and $w(1), w(2), \ldots$. 
It is  known 
\cite[Theorem 4.6]{MR2245368}  %%% CSP 
\cite{MR2160320} %%%AUTHOR = {Gnedin, A. and Pitman, J.}, TITLE = {Exchangeable {G}ibbs partitions and {S}tirling triangles},
that the only $\EPPF$s of this form which are produced by sampling from some frequencies $P_j$ 
that are strictly positive for all $j$ are those obtained by taking 
\begin{equation}
\label{gibbsw}
w(n) =  (1-\alpha)_{n-1} \mbox{  for some } 0 \le \alpha < 1,
\end{equation}
with $(x)_n:=\Gamma(x+n)/\Gamma(x)$ the rising factorial.
For such weights $w(\cdot)$ it is easy to see 
\cite{MR2160320} %%%AUTHOR = {Gnedin, A. and Pitman, J.}, TITLE = {Exchangeable {G}ibbs partitions and {S}tirling triangles},
that $V_{k:n}$ must satisfy the consistency relation 
\begin{equation}
\label{consistency}
V_{k:n}=(n-k\alpha) V_{k:n+1}+V_{k+1:n+1}\qquad(1\le k\le n<\infty).
\end{equation}
Following 
\cite[Definition 3]{MR2160320} %%%AUTHOR = {Gnedin, A. and Pitman, J.}, TITLE = {Exchangeable {G}ibbs partitions and {S}tirling triangles},
we call an exchangeable partition of positive integers  with $\EPPF$ of form \eqref{gibbseppf} with $w$ weights given by \eqref{gibbsw} and $V$ weights subject to \eqref{consistency}
a {\em Gibbs\/$(\alpha)$ partition}. For a given $\alpha \in [0,1)$, the collection of all arrays of 
nonnegative weights $V_{k:n}$ satisfying \eqref{consistency} is a convex set
\cite[Theorem 4.6]{MR2245368}  %%% CSP 
\cite{MR2160320}. %%%AUTHOR = {Gnedin, A. and Pitman, J.}, TITLE = {Exchangeable {G}ibbs partitions and {S}tirling triangles},
For each $\alpha$ there is a one-parameter family of extreme weights. These are indexed by $\theta \ge 0$  for $\alpha = 0$, and 
$0 < \alpha < 1$ by another parameter $\ell \ge 0$, called the {\em $\alpha$-diversity} of the associated random partition in \cite{MR2245368}. %%% CSP 
In both cases, by general convex analysis, every consistent family of weights $V_{k:n}$ admits a unique integral representation over this one-parameter family of extreme weight arrays.

For each 
$\alpha \in [0,1)$ and $\theta > - \alpha$ there is the distinguished family of weights
\begin{equation}
\label{althweights}
V_{k:n}(\alpha,\theta):=  \frac{1}{(1 + \theta)_{n-1} } \prod_{i=1}^{k-1} ( \theta + i \alpha).
\end{equation}
It is easily checked that \eqref{consistency} holds for these special weights $V_{k:n} = V_{k:n}(\alpha,\theta)$, so
\eqref{althweights} provides an instance of Gibbs($\alpha$) exchangeable partition. It is known
\cite{MR2160323}
that such $V_{k:n}$ are the only weights of form $V_k/c_n$ for some positive sequences $V_\bullet$ and $c_\bullet$, and that the weights
\eqref{althweights} produce the $\EPPF$ of a random partition of positive integers whose frequencies in order of appearance have the $\GEM(\alpha,\theta)$ distribution
described in the Introduction \cite{MR2245368}. %%% CSP 

Gibbs partitions were introduced in \cite{MR2004330} %% Poisson-Kingman partitions
and further developed in \cite{MR2160320}, and received much attention in recent probabilistic and statistical literature. 
A wide range of Gibbs $\EPPF$ in terms of special functions can be found in \cite{james-lau-gibbs}. %%author = {Ho, M-W, James, L.F. and Lau, J.W. }.  year = {2007}, title = {Gibbs partitions (EPPF’s) derived from a stable subordinator are Fox H - And Meijer G - Transforms},
Not trying to provide a full review of the literature, we mention papers 
\cite{MR2479467} %% 2008 AUTHOR = {Cerquetti, Annalisa}, TITLE = {On a {G}ibbs characterization of normalized generalized {G}amma processes}, JOURNAL = {Statist. Probab. Lett.}, FJOURNAL = {Statistics \& Probability Letters},
\cite{MR2469329} %% 2008 AUTHOR\ \=\ \{Lijoi\,\ Antonio\ and\ Pr\{\\\"u\}nster\,\ Igor\ and\ Walker\,\ Stephen\ G.\}\,\ \ \ \ \ \ TITLE\ \=\ \{Investigating\ nonparametric\ priors\ with\ \{G\}ibbs\ structure\}\, 
\cite{cerquetti2009generalized} %% 2009 title={A generalized sequential construction of exchangeable Gibbs partitions with application}, author={Cerquetti, Annalisa},
\cite{MR3035269} %% 2013 AUTHOR = {Cerquetti, Annalisa}, TITLE = {Marginals of multivariate {G}ibbs distributions with applications in {B}ayesian species sampling}, JOURNAL = {Electron. J. Stat.},
\cite{MR3051207} %%% 2013 AUTHOR = {Cerquetti, Annalisa}, TITLE = {Some contributions to the theory of conditional {G}ibbs partitions},
\cite{MR3252637} %% 2014 title={Posterior analysis of rare variants in Gibbs-type species sampling models}, author={Cesari, Oriana and Favaro, Stefano and Nipoti, Bernardo},
\cite{MR3439188} %%%, 2015 AUTHOR = {Favaro, Stefano and James, Lancelot F.}, TITLE = {A note on nonparametric inference for species variety with {G}ibbs-type priors},
\cite{de2015gibbs} %% 2015 title={Are Gibbs-type priors the most natural generalization of the Dirichlet process?}, 
which deal with various statistical applications of Gibbs-type priors.  
Interpretations of Gibbs partitions in terms of records were developed in \cite{MR2336601}.   %%, AUTHOR = {Griffiths, Robert C. and Span{\`o}, Dario}, TITLE = {Record indices and age-ordered frequencies in exchangeable {G}ibbs partitions},
Some features of the construction we use in Section \ref{sec:related} are interpreted 
in terms of Bayesian inference for species sampling in 
\cite{MR2434179} %%, 2008 AUTHOR = {Lijoi, Antonio and Pr{\"u}nster, Igor and Walker, Stephen G.}, TITLE = {Bayesian nonparametric estimators derived from conditional {G}ibbs structures},
and \cite{MR3322311}. %%title={Looking-backward probabilities for Gibbs-type exchangeable random partitions}, author={Bacallado, Sergio and Favaro, Stefano and Trippa, Lorenzo and others},

\subsection{The Chinese restaurant process}
\label{sec:crp}

The {\em Chinese Restaurant Process }($\CRP$)  \cite[\S 3.1]{MR2245368}  %%% CSP 
provides an intuitive expression of various notions in sampling from random discrete distributions,
with its metaphorical customers arriving to be seated at tables in the restaurant, with interpretations in various contexts,  of
\begin{itemize}
\item {\em customers} corresponding to individuals/tokens/elements;

\item {\em tables} corresponding  to values/alleles/species/types/cycles/clusters/blocks/intervals.
\end{itemize}
In the basic $\CRP$ as described in \cite[\S 3.1]{MR2245368}  %%% CSP 
an exchangeable random partition of positive integers is constructed sequentially.
Starting from a first customer assigned to table $1$, after $n$ customers have been assigned to some $k$ tables labeled by $1,2, \ldots,k$ in appearance-order,
with say $n_i$ customers seated at table $i$, for $1 \le i \le k$, 
there is a probabilistic rule for assigning customer $n+1$ either to one of the $k$ tables 
already occupied 
or to a new table.
In the ecological context of species sampling, the customers are individuals and assigning a new customer to one of previously occupied tables corresponds
to observing an individual of some previously seen species, while introducing a new table corresponds to sampling an individual of some species previously unseen.
In this basic $\CRP$, the relative frequencies of customers occupying the tables in order of appearance  converge back to $P_\bullet^*$, the size-biased permutation of population frequencies,
as in \eqref{asconv}.

In the context of Gibbs$(\alpha)$ partitions with $\EPPF$ \eqref{gibbseppf}, 
given appearance-ordered frequencies $(n_1, \ldots, n_k)$ in a sample of size $n$, the
appearance-ordered frequencies in a sample of size $n+1$ are obtained by
\begin{itemize}
\item either adding the frequency $n_{k+1}=1$ (discovering a new species) with probability 
\begin{equation}
\label{gibbspnk}
p_{k:n}:=\frac{p(n_1,\dots,n_k,1)}{p(n_1,\dots,n_k)}=\frac{V_{k+1:n+1}}{V_{k:n}},\qquad(1\le k\le n);
\end{equation}
\item or, for some $i\in[k]$, incrementing $n_{i}$ by $1$ (new observation is the $i$th species in order of appearance) with probability 
\begin{align}
\label{app:sb}
\notag
\frac{p(n_1,\dots,n_i+1,\dots,n_k)}{p(n_1,\dots,n_k)}
&{}= (1 - p_{k:n}) \frac{ n_i - \alpha}{ n-k \alpha } \\
&{}= (1 - p_{k:n}) h_\alpha(n_i, \ldots, n_k) \prod_{j=1}^{i-1} \bigl[ 1 -  h_\alpha(n_j, \ldots, n_k)  \bigr]
\end{align}
where for a composition $(n_1, \ldots, n_k)$ of $n$
\begin{equation}
\label{app:sb1}
h_\alpha(n_1, \ldots, n_k)  := \frac{ n_1 - \alpha } {n_1 - \alpha + \cdots + n_k - \alpha } =  \frac{ n_1 - \alpha }{ n - k \alpha }
\end{equation}
is the probability of choosing the first cluster in a $\sizem$-biased pick from $k$ distinct clusters of sizes $n_1, \ldots, n_k$.
\end{itemize}
The consistency relation \eqref{consistency} ensures that these probabilities sum to $1$.  
The second, product form in \eqref{app:sb} emphasizes the idea that a single cluster can be chosen from the existing clusters of sizes $(n_1, \ldots, n_k)$,
in a $\sizem$-biased fashion, by a succession of $\sizem$-biased choices, the first to decide if it is the first cluster or not, if not whether it is the second
cluster, and so on.
In particular, for $\GEM(\alpha,\theta)$ frequencies defined by \eqref{althweights} one has
\begin{equation}
\label{gem:disc}
p_{k:n} = \frac{ \theta + k \alpha} { \theta + n } \qquad\mbox{and}\qquad \frac{p(n_1,\dots,n_i+1,\dots,n_k)}{p(n_1,\dots,n_k)}=\frac{n_i-\alpha}{\theta+n}\,.
\end{equation}
Applying this procedure step by step leads
to appearance-ordered tables, with relative frequencies converging to size-biased permutation of the Gibbs$(\alpha)$ probabilities, as in \eqref{asconv},
which have the same distribution as the original frequencies just in the $\ISBP$ $\GEM(\alpha,\theta)$ case.

\section{Main results}
\label{sec:main}

Our central result is the the following more refined version of Theorem \ref{thm:sbm}:
%\begin{corollary}
\begin{theorem}
\label{thm:gibbs}
Fix $0 \le \alpha < 1$. Suppose that $X_1, \ldots, X_n$  is a random sample from 
$\Pbul$ which is the size-biased presentation of limit frequencies of a Gibbs$(\alpha)$ partition of positive integers. Let $p$ be the $\EPPF$,
as in \eqref{gibbseppf}--\eqref{gibbsw}, corresponding to \eqref{althweights} for $\Pbul$ with $\GEM(\alpha,\theta)$ distribution.
Then the composition probability function of
$\Nright_{\bullet:n}$, the sequence of sizes of clusters of $X$-values in increasing order of those values, is given by the formula
\begin{align}
\label{rightform} 
\BP[ \Nright_{\bullet:n} = (n_1,\ldots, n_k) ] = {n \choose n_1, \ldots , n_k } \left( \prod_{i=1}^k  \frac{n_i - \alpha} {n_i - \alpha + \cdots + n_k - \alpha} \right) \, p(n_1, \ldots, n_k)
\end{align}
for each composition $(n_1,\ldots, n_k)$ of $n$.
Moreover,
\begin{equation}
\label{sbmagic-repeat}
\mbox{ $\Nright_{\bullet:n}$ is a $\sizem$-biased permutation of $\Nstar_{\bullet:n}$}
\end{equation}
where the composition probability function of $\Nstar_{\bullet:n}$  is given by the right side of \eqref{rightform} with the $\sizem$-biasing product replaced by the ordinary
size-biasing product with every $\alpha$ replaced by $0$, as in the known formula \eqref{count:dt}.
\end{theorem}
Our proof of this result makes use of the following key lemma:

\begin{lemma}
{\em \cite[Proposition 3.1]{2017gem}}
\label{lmm:key}
Consider an exchangeable random partition $\Pi_\infty:= (\Pi_n)$ of the set $\BN$
of positive integers, with proper frequencies, meaning that all clusters of\/ $\Pi_\infty$ are infinite almost surely. Fix $n \ge 1$.
Let $\CC_1$ be the cluster of\/ $\Pi_\infty$ containing $n+1$, and for $k \ge 1$, given that $\cup_{i=1}^k \CC_i \ne \BN$, let $\CC_{k+1}$ be the cluster of\/ $\Pi_\infty$ containing the least $m > n + 1$ with
$m \notin \cup_{i=1}^k \CC_i$. Define $\Pbul$ by setting $P_k$ equal to the almost sure limiting relative frequency of $\CC_k$, and
let $X_i:= j$ iff\/ $i \in \CC_j$.  Then
\begin{itemize}
\item $\Pbul$ is a size-biased ordering of frequencies of\/ $\Pi_\infty$;
\item $(X_1, \ldots, X_n)$ is a sample of size $n$ from $\Pbul$;
\item the partition of $[n]$ generated by $(X_1, \ldots, X_n)$ is\/ $\Pi_n$.
\end{itemize}
\end{lemma}

This key lemma has a ``now you see it, now you don't'' quality, depending on what metaphor is used for the intuitive description of $\Pi_\infty$.
Here is one way to see that the construction works. Regard the exchangeable partition $\Pi_\infty$ as being generated by Kingman's paintbox construction, using random frequencies to create
an interval partition of $[0,1]$ into component intervals whose lengths in size-biased random order have the 
required 
distribution of $\Pbul$.
Let $U$ be a sequence of i.i.d.\ uniform $[0,1]$ random variables independent of the interval partition, and let $\Pi_\infty$ be the
random partition of $\BN$ whose clusters are the equivalence classes for the random equivalence relation $i \sim j$ iff either $i=j$ or $U_i$ and $U_j$ fall in the
same component of the interval partition.
Use the order in which $U_{n+1}, U_{n+2}, \ldots$ discover the component intervals to label them by $1, 2, \ldots$, and define $\Pbul$ by this labeling of interval lengths. 
Finally, let $X_i$ for $i \in [n]$ be the numerical label of the interval component containing $U_i$. Then the conclusions of the lemma should be intuitively clear.

Our proof of Theorem \ref{thm:gibbs} will be expressed in terms of another metaphor for exchangeable random partitions,
the Chinese Restaurant construction of $\Pi_\infty$.
Call the first $n$ customers to enter the restaurant
the {\em primary customers} and customers $n+1, n+2, \ldots$ the {\em secondary} customers. 
Then $\CC_1$, $\CC_2$, $\ldots$ is the list of clusters of $\Pi_\infty$ {\em in order of their discovery by secondary customers},
and $\Pbul$ is the listing of frequencies of these clusters of $\Pi_\infty$ in that order of discovery by secondary customers.
Compared to the usual listing of clusters in order of appearance, this is just a relabeling of tables. Each table in the restaurant is assigned
a new label, with label $1$ for the table at which customer $n+1$ is seated, label $2$ for the next table discovered by one of the secondary customers, and so on.
By some almost surely finite random time, the first $K_n$ tables at which the primary customers were seated will all have been discovered by secondary customers.
At that random time, the values $X_1, \ldots, X_n$ are assigned to the primary customers, with $X_i = j$ if customer $i$ is seated at the $j$th table in
order of discovery by the secondary customers. With this metaphor, the fact that $X_1, \ldots, X_n$ generates $\Pi_n$, the partition of $[n]$
defined by the original seating  plan in the Chinese Restaurant, is completely obvious. That $X_1, \ldots, X_n$ is a sample of size $n$ from $\Pbul$ is less obvious, 
but nonetheless true.

\begin{proof}[Proof of Theorem \ref{thm:gibbs}]
It is enough to show \eqref{sbmagic-repeat} for any particular representation
of a sample $X_1, \ldots X_n$ from $\Pbul$.
For this purpose, we take $X_1, \ldots X_n$ to be constructed as in Lemma \ref{lmm:key},
and use the Chinese Restaurant metaphor.
According to \eqref{app:sb}, in a Gibbs$(\alpha$) $\CRP$
at any stage $m > n$ in the process of rediscovery of the initial tables by secondary customers, 
given that $\Nstar_{\bullet:n} = (n_1, \ldots, n_k)$ say, and
given that up to stage $m$ some non-empty subset of tables $S \subseteq [k]$ remains undiscovered,
and given also that individual $m+1$ sits at one of these tables, 
that table is table $i \in S$ with probability $(n_i - \alpha)/\sum_{s \in S} (n_s -\alpha)$. That is just a $\sizem$-biased assignment of customer $m+1$ 
to one of the remaining tables.  The conditional distribution of $\Nright_{\bullet:n}$  given $\Nstar_{\bullet:n} = (n_1, \ldots, n_k)$ is therefore that of a
$\sizem$-biased random permutation of $(n_1, \ldots, n_k)$.
This proves \eqref{sbmagic-repeat}, and the sampling formula \eqref{rightform} is read from Lemma \ref{lmm:order}.
\end{proof}

The simplicity of these descriptions of value-ordered frequencies in sampling from $\GEM(\alpha,\theta)$, and  Gibbs$(\alpha)$ models in general, suggests
there should be some embellishment of the $\CRP$ generating appearance-ordered frequencies as
described in Section \ref{sec:crp},
%\cite[xxx]{MR2245368},  %%% CSP 
in which both the appearance-order and  value-order of the sample are generated sequentially, in an entirely combinatorial way, that is distributionally 
equivalent to the model of sampling from an infinite list of frequencies. 
Such additional structure of sampling in an environment with totally ordered clusters, treated in
\cite{MR1104078} %%%, TITLE = {Consistent ordered sampling distributions: characterization and convergence},
\cite{MR1457625} %% AUTHOR = {Gnedin, Alexander V.}, TITLE = {The representation of composition structures},
\cite{MR2122798} %%AUTHOR = {Gnedin, Alexander and Pitman, Jim}, TITLE = {Regenerative composition structures},
and  developed here in Corollary \ref{crl:ocrp},
is well accommodated by an {\em Ordered Chinese Restaurant Process} ($\OCRP$).
This is the usual $\CRP$, 
with a sequentially developing total order of tables, as proposed in \cite{MR2275248}.  %%, AUTHOR = {James, Lancelot F.}, TITLE = {Poisson calculus for spatial neutral to the right processes},
Here the order of tables is taken to be the value-order, although any other order of tables can be treated in a similar fashion.
It is assumed inductively that after $n$ customers have arrived they are seated at some $k$ tables which are 
placed from left to right by order of values in the sample, and a new customer is seated either to some previously occupied table or
to a new table which is placed in one of $k+1$ possible places relative to the old tables.
Technically, the {\em state} of the restaurant after $n$ customers have been seated represents an 
{\em ordered partition of the set of\/ $n$ customers labeled by }$[n]$. 
Customer $n+1$ arrives with a value $X_{n+1}$ and occupies the table where previous customers with the same value were seated, if any, or a new
table if this value appears for the first time, and this new table is placed between tables with lower and higher values than $X_{n+1}$, or at the appropriate 
end of the line of tables if the value $X_{n+1}$ is extreme compared to the values of previous customers. 
Implicit then in the state is the ordering of tables by appearance which can be restored by sorting the tables in order of the least customer
number. If just the value-ordered frequencies of the occupied tables are given instead of the ordered partition of the set $[n]$ then this information is lost. 
But due to exchangeability, the conditional distribution of the appearance order given value-ordered frequencies is a size-biased permutation
of these frequencies.

It turns out that the above procedure specialized to Gibbs$(\alpha)$ partitions with value-ordered frequencies can be described
in a way quite similar to the basic $\CRP$ explained in Section \ref{sec:crp}.  We summarize it in the following Corollary of Theorem~\ref{thm:gibbs}.

\begin{corollary}
\label{crl:ocrp}
In sampling from the limit frequencies of any Gibbs$(\alpha)$ model in size-biased order,
with associated discovery probabilities $p_{k:n}$ as in \eqref{gibbspnk},
the sequential development of value-ordered sampling frequencies is as follows.
Given frequencies $(n_1, \ldots, n_k)$ in value-order in a sample of size $n$, the value-ordered frequencies in the
sample of size $n+1$ are obtained by
\begin{itemize}
\item either, for some $j\in[k+1]$, putting a $1$ into $(n_1, \ldots, n_k)$ at the $j$th of $k+1$ possible places 
(new value not present previously and of rank $j$ in the updated value order),  to create frequencies
$(1,n_1, \ldots, n_k)$,  $(n_1, 1, \ldots, n_k)$, $\ldots$, $(n_1, \ldots, n_k, 1)$ as the case may be, with probabilities
\begin{equation}
\label{freq:sb}
p_{k:n}  \,   h_\alpha(1,n_j, \ldots, n_k) \prod_{i=1}^{j-1} \bigl[ 1 -  h_\alpha(1, n_i, \ldots, n_k) \bigr];
\end{equation}
\item or, for some $j\in[k]$, incrementing $n_{j}$ by $1$ (new value of rank $j$ both in the previous and in the updated value ordering) with probabilities
\begin{equation}
\label{freq1:sb}
(1 - p_{k:n}) h_\alpha(n_j + 1, n_{j+1},\ldots, n_k) \prod_{i=1}^{j-1} \bigl[ 1 -  h_\alpha(n_i +1, n_{i+1},\ldots, n_k) \bigr]
\end{equation}
for $h_\alpha(n_1, \ldots, n_m):= (n_1 - \alpha)/(n_1 - \alpha + \cdots + n_m - \alpha)$ as in \eqref{app:sb1}.
\end{itemize}
\end{corollary}

This corollary is much simpler than similar descriptions of the development
of value-ordered sampling frequencies for a $\RAM$ with i.i.d.\ factors  provided by 
Gnedin and Pitman \cite{MR2122798} %%AUTHOR = {Gnedin, Alexander and Pitman, Jim}, TITLE = {Regenerative composition structures},
and James \cite{MR2275248},  %%, AUTHOR = {James, Lancelot F.}, TITLE = {Poisson calculus for spatial neutral to the right processes},
even in the case of sampling from $\GEM(0,\theta)$, when it can be checked that  Corollary \ref{crl:ocrp} is consistent with results in these sources. 
What is remarkable and unexpected about these results is that it seems extremely difficult to provide any comparably simple descriptions
of the value-ordered frequencies in sampling from a more general  $\RAM$ with independent but not identically distributed factors.
Our arguments make essential use of both the assumed size-biased order of the Gibbs$(\alpha)$ frequencies,
and the sequential description of Gibbs$(\alpha)$ sampling frequencies in appearance order, discussed above.

\begin{proof}[Proof of Corollary \ref{crl:ocrp}]
Suppose that after $n\ge1$ steps of the $\OCRP$ the value-ordered frequencies are $(n_1,\dots,n_k)$, with $n=\sum_{i=1}^k n_i$.
Given that event, according to \eqref{rightform} and Lemmas \ref{lmm:oeppf} and \ref{lmm:order}
the probability that a new customer occupies some new table which is placed in $j$th of $k+1$
possible positions is
\[
\frac{\tilp(n_1,\dots,n_{j-1},1,n_j,\dots,n_k)}{\tilp(n_1,\dots,n_k)}
=\frac{\psize(n_1,\dots,n_{j-1},1,n_j,\dots,n_k)}{\psize(n_1,\dots,n_k)}\,\frac{V_{k+1:n+1}}{V_{k:n}}
\]
for $\psize$ associated with the pseudo-size function $s(n)=n-\alpha$ as in \eqref{psize:def}, where
the second fraction on the right is the ratio of $\EPPF$s \eqref{gibbseppf}, because products of $w$ weights \eqref{gibbsw} cancel. 
By \eqref{gibbspnk} this second fraction is exactly $p_{k:n}$ for the Gibbs$(\alpha)$ partitions. Hence it remains to 
notice that the ratio of $\psize$ functions can be written down in the stick-breaking form \eqref{freq:sb}.

Similarly, given the frequencies $(n_1,\dots,n_k)$ in value-order after $n$ steps, a new customer is seated at the existing table $j$
with probability
\[
\frac{\pcomp(n_1,\dots,n_j+1,\dots,n_k)}{\pcomp(n_1,\dots,n_k)}
=\frac{\psize(n_1,\dots,n_j+1,\dots,n_k)}{\psize(n_1,\dots,n_k)}\,\frac{V_{k:n+1}w(n_j+1)}{V_{k:n}w(n_j)}\,.
\]
From \eqref{gibbsw}, \eqref{consistency}, \eqref{gibbspnk} and \eqref{psize:def} it follows that 
\begin{gather*}
\frac{V_{k:n+1}w(n_j+1)}{V_{k:n}w(n_j)}=\frac{n_j-\alpha}{n-k\alpha}(1-p_{k:n}),\\
\frac{\psize(n_1,\dots,n_j+1,\dots,n_k)}{\psize(n_1,\dots,n_k)}
=\frac{n_j+1-\alpha}{n_j-\alpha}\prod_{i=1}^j\frac{n_i-\alpha+\dots+n_k-\alpha}{n_i-\alpha+\dots+n_k-\alpha+1}
\end{gather*}
and the stick-breaking representation \eqref{freq1:sb} is just a rearrangement of the product of
the right-hand sides above. 
\end{proof}

Two comments on the above argument. 

\noindent
$\bullet$
Once the Bayes ratios have  been calculated as indicated, the conditional independence asserted in Corollary \ref{crl:cndind} is obvious by inspection of the
formulas. But this conditional independence does not seem at all obvious otherwise. Especially because the Bayes calculations show that the value-ordered frequencies do affect the
probabilities of adding to old clusters, it does not seem at all clear why they might not also affect the probability of discovering a new species, in some way more complex than
just through the total number of clusters. Even for sampling from $\GEM(0,\theta)$ this does not seem obvious.

\noindent
$\bullet$
The elementary algebra of cancellation in the Bayes calculations used to prove Corollary \ref{crl:ocrp} can be easily  used to show that the $\OCRP$ defined by that
Corollary \ref{crl:ocrp} gives an ordered exchangeable  partition  of positive integers, without any assumption that it is derived by the value-orders in successive sampling.
In view of the general representation theorem for such
an exchangeable $\OCRP$ due to 
Gnedin \cite{MR1457625} %% AUTHOR = {Gnedin, Alexander V.}, TITLE = {The representation of composition structures},
it follows that this $\OCRP$ must in fact be derived from value-order generated by some exchangeable sequence $X_1,X_2, \ldots$, which is a sample from some random discrete distribution $F$ on the line,
whose size-biased atoms have the Gibbs$(\alpha)$ distribution determined by the discovery probabilities, because the distribution of partitions of $n$ is built into the dynamics of the restaurant. 
Even for $\GEM(\alpha,\theta)$ it seems far from obvious from this approach why the atoms of $F$ are simply those of Gibbs$(\alpha)$ in their usual order, which is at the heart of what Theorems \ref{thm:sbm} and \ref{thm:gibbs}
are saying.
But this approach might be used in combination with some other means of identifying $F$ to provide an alternate proof of Theorem \ref{thm:sbm}.

\section{Related calculations}
\label{sec:related}

%\subsection{Brute force proof of Theorem \ref{thm:sbm}}
\subsection{Inductive proof of Theorem \ref{thm:sbm}}

For a general RAM, a decomposition over the minimal value $m$ of the sample, which is repeated $n_1$ times if the value-ordered counts are $(n_1, \ldots, n_k)$,
leads to the following recursive formula\footnote{We thank the anonymous referee for this observation.}: for $k\ge2$
\begin{multline}
\label{recursive}
\BP[\Nright_{\bullet:n}=(n_1,\dots,n_k)]\\=\binom{n}{n_1}\sum_{m=1}^\infty\BE\Bigl[H_m^{n_1}(1-H_m)^{n-n_1}\prod_{i=1}^{m-1}(1-H_i)^n\Bigr]
\BP^{(m)}[\Nright_{\bullet:n}=(n_2,\dots,n_k)],
\end{multline}
where $\BP^{(m)}$ refers to the probability in a generally different RAM, that is one generated by $(H_{m+1},H_{m+2},\dots)$ instead of
$(H_1,H_2,\dots)$  in \eqref{stickbreak}.  So only for RAMs with i.i.d.\ factors $H_i$ this is indeed a recursion, but for $\GEM(\alpha,\theta)$
\eqref{recursive} connects the probabilities of value-ordered counts for different parameters: 
\begin{multline}
\label{recursive-gem}
\BP_{\alpha,\theta}[\Nright_{\bullet:n}=(n_1,\dots,n_k)]\\=\binom{n}{n_1}\sum_{m=1}^\infty\BE\Bigl[H_m^{n_1}(1-H_m)^{n-n_1}\prod_{i=1}^{m-1}(1-H_i)^n\Bigr]
\BP_{\alpha,\theta+m\alpha}[\Nright_{\bullet:n}=(n_2,\dots,n_k)].
\end{multline}
This leads to a direct proof of Theorem \ref{thm:sbm} by induction on the number $k$ of distinct values in the sample, which is outlined below.

We need to show that, with $n=n_1+\ldots+n_k$, 
%For the induction step, suppose that for some $k\ge 1$ and
\begin{equation}\label{DTrightarrow}
\BP_{\alpha,\theta}[\Nright_{\bullet:n}=(n_1,\dots,n_k)]=
\frac{n!\bigl(\tfrac\theta\alpha+1\bigr)_{k-1}\alpha^{k-1}}{(\theta+1)_{n-1}}
\prod_{\ell=1}^k\frac{(1-\alpha)_{n_\ell}}{n_\ell!(n_\ell+\dots+n_k-(k-\ell+1)\alpha)}
\end{equation}
which is \eqref{rightform} for the $\GEM(\alpha,\theta)$ $\EPPF$ $p$ given by
\eqref{gibbseppf} with $w$ weights \eqref{gibbsw} and $V$ weights \eqref{althweights}.
Comparing it to the well-known formula  \cite[(3.6)]{MR2245368}.  %%% CSP 
\begin{equation}
\BP_{\alpha,\theta}[\Nstar_{\bullet:n}=(n_1,\dots,n_k)]=
\frac{n!(\tfrac{\theta}{\alpha}+1)_{k-1}\alpha^{k-1}}{(\theta+1)_{n-1}}\prod_{\ell=1}^k\frac{(1-\alpha)_{n_\ell-1}}{n_\ell!}
\end{equation}
shows that $\Nright_{\bullet:n}$ is a $\sizem$-biased permutation of $\Nstar_{\bullet:n}$.

In order to evaluate \eqref{recursive-gem} note that
\begin{equation}
\label{Himoments}
\BE_{\alpha,\theta}H_i^r(1-H_i)^s 
%&
=\frac{B(1-\alpha+r,\theta+i\alpha+s)}{B(1-\alpha,\theta+i\alpha)}
%\\
%&=\frac{\Gamma(1-\alpha+n)\Gamma(\theta+i\alpha+\ell)\Gamma(\theta+(i-1)\alpha+1)}{\Gamma(\theta+(i-1)\alpha+n+\ell+1)\Gamma(1-\alpha)\Gamma(\theta+i\alpha)}\\
%&
=\frac{(1-\alpha)_r(\theta+i\alpha)_s}{(\theta+(i-1)\alpha+1)_{r+s}}.
\end{equation}
For samples with just one value repeated $n$ times it is well known and easy to calculate using \eqref{Himoments} that
\begin{align*}
\BP_{\alpha,\theta}[\Nright_{\bullet:n}=(n)]&{}=\sum_{m=1}^\infty \BE_{\alpha,\theta}P_m^n
=\sum_{m=1}^\infty \BE_{\alpha,\theta}H_m^n \prod_{i=1}^{m-1}\BE_{\alpha,\theta}(1-H_i)^n \\
&{}=\sum_{m=1}^\infty \frac{(1-\alpha)_n}{(\theta+(m-1)\alpha+1)_n}\prod_{i=1}^{m-1}\frac{(\theta+i\alpha)_n}{(\theta+(i-1)\alpha+1)}\\
&{}=\frac{(1-\alpha)_n}{(\theta+1)_n}\sum_{m=1}^\infty\prod_{i=1}^{m-1}\frac{\theta +i\alpha}{\theta+i\alpha+n}.
\end{align*}
The sum in the last line is the evaluation of the hypergeometric function ${}_2F_1(1,\tfrac{\theta+\alpha}{\alpha};\tfrac{\theta+\alpha+n}{\alpha};1)$, and
since for $b>a+1$ one has 
\cite[\S14.11]{MR1424469} %%%    AUTHOR = {Whittaker, E. T. and Watson, G. N.},   TITLE = {A course of modern analysis},}
\begin{equation}\label{hyper}
\sum_{j=0}^{\infty}\frac{(a)_j}{(b)_j}={}_2F_1(1,a;b;1)=\frac{b-1}{b-a-1}
\end{equation}
it finally gives 
\begin{equation}\label{DTn}
\BP_{\alpha,\theta}[\Nright_{\bullet:n}=(n)]=\frac{(1-\alpha)_{n-1}}{(\theta+1)_{n-1}}
\end{equation}
in accordance with \eqref{DTrightarrow}.
This is the induction base.

Now suppose that \eqref{DTrightarrow} is true for some $k$ and check that it is also true with $k+1$ instead of $k$.
Let $n=n_1+\ldots+n_{k+1}$, then by \eqref{recursive-gem}, \eqref{Himoments} and the induction assumption \eqref{DTrightarrow}
\begin{align*}
&\BP_{\alpha,\theta}[\Nright_{\bullet:n}=(n_1,\dots,n_{k+1})]\\
&\qquad=\binom{n}{n_1}\sum_{m=1}^\infty\BE\Bigl[H_m^{n_1}(1-H_m)^{n-n_1}\prod_{i=1}^{m-1}(1-H_i)^n\Bigr]
\BP_{\alpha,\theta+m\alpha}[\Nright_{\bullet:n-n_1}=(n_2,\dots,n_k)]\\
&\qquad=\frac{n!}{n_1!}\sum_{m=1}^\infty \frac{(1-\alpha)_{n_1}(\theta+m\alpha)_{n-n_1}}{(\theta+(m-1)\alpha+1)_n}\prod_{i=1}^{m-1}\frac{(\theta+i\alpha)_n}{(\theta+(i-1)\alpha+1)_n}\times\\
&\qquad\qquad\qquad\times\frac{\bigl(\tfrac{\theta+m\alpha}\alpha+1\bigr)_{k-1}\alpha^{k-1}}{(\theta+m\alpha+1)_{n-n_1-1}}
\prod_{\ell=2}^{k+1}\frac{(1-\alpha)_{n_\ell}}{n_\ell!(n_\ell+\dots+n_{k+1}-(k-\ell+2)\alpha)}\\
&\qquad=\frac{n!(n-(k+1)\alpha)\alpha^{k}}{(\theta+1)_n}\prod_{\ell=1}^{k+1}\frac{(1-\alpha)_{n_\ell}}{n_\ell!(n_\ell+\dots+n_{k+1}-(k-\ell+2)\alpha)}%\times\\
%&\qquad\qquad\qquad 
\sum_{m=1}^\infty \frac{(\tfrac\theta\alpha+1)_{k+m-1}}{(\tfrac{\theta+n}{\alpha}+1)_{m-1}}.
%\prod_{i=1}^{m-1}\frac{\theta+i\alpha}{\theta+i\alpha+n}
\end{align*}
Writing $(\tfrac{\theta}{\alpha}+1)_{k+m-1}=(\tfrac{\theta}{\alpha}+1)_k(\tfrac{\theta}{\alpha}+k+1)_{m-1}$ and using \eqref{hyper} allows
to calculate 
\[
\sum_{m=1}^\infty \frac{(\tfrac\theta\alpha+1)_{k+m-1}}{(\tfrac{\theta+n}{\alpha}+1)_{m-1}}
=(\tfrac{\theta}{\alpha}+1)_k\frac{\theta+n}{n-(k+1)\alpha}
\]
which gives \eqref{DTrightarrow} with $k$ replaced by $k+1$, as desired.

\subsection{Some checks on Corollary \ref{crl:ocrp}}

As the result of Corollary \ref{crl:ocrp} is a new and not obvious property of Gibbs$(\alpha)$ partitions, even in the heavily studied case $\alpha = 0$ of
a $\GEM(0,\theta)$ partition, this section offers some checks on the result by different approaches. We are able to do this for $\alpha = 0$,
but providing any significant checks on the result for $0 < \alpha < 1$ remains a challenging problem.

We start with some general identity in distribution which is a consequence of Lemma \ref{lmm:key}. Let $(X_1,X_2,\dots)$ be a sample
from random discrete distribution $\Pbul$, and consider
the sequence of indicators $\Delta_n$ and $L_n$ where $\Delta_n:= K_n - K_{n-1}$  is the indicator 
of discovery of a new value at step $n$, 
%%%as in \eqref{times} with upper index $n=0$ omitted, 
and $L_n$ is the indicator of placement at the extreme left, i.e.\ that the new value is less than
all previous values. Obviously $0 \le L_n \le \Delta_n$. 

Let the new values be discovered in the random times 
\begin{equation}
\{ 1 = M_1 < M_2 < \cdots \} := \{n \ge 1 : \Delta_n  = 1 \}.
\end{equation}
Let $X$ be the index in this sequence of the time value $1$ first appears in the sample, that is
\begin{equation}
M_X = \min\{m:X_m=1\}. 
\end{equation}

\begin{corollary}\label{cor:X=X_1} 
Suppose that frequencies $\Pbul$ are in size-biased random order. Then $X$ has the same distribution as the first sample $X_1$.
\end{corollary}

\begin{proof} We can think of sampling from any realization of frequencies $\Pbul$. Consider Kingman's paintbox construction of
Section \ref{sec:partitionsfromsamples} and suppose that $U_1$ hits some interval $\mathcal{I}$.
Take $\Pbul$ as in Lemma~\ref{lmm:key} with $n=1$, then they are in size-biased random order, and $X_1$ as defined in the Lemma is
the number of clusters in $\Pi_\infty$ restricted to $\{2,\dots,L\}$, where $L$ is the random time when the sequence $U_2,U_3,\dots$ hits
$\mathcal{I}$. 
On the other hand, suppose that $M_1,M_2,\dots$ and $X$ are produced from the sample from $\Pbul^*\ed \Pbul$, where $P_1^*$ is the
size-biased pick defined as the length of $\mathcal{I}$. Then the order of other frequencies of $\Pbul^*$ is irrelevant to
the definition of $X$, and $X=X_1$ almost surely,
hence the result.
\end{proof}

For the Gibbs$(\alpha)$ partition Corollary \ref{crl:ocrp} gives 
us a very different way to compute the law of $X$:

\begin{corollary}
For the Gibbs$(\alpha)$ partition generated by a sample $X_1, X_2, \ldots$ from Gibbs$(\alpha)$ frequencies $\Pbul$ in size-biased order, let
\begin{equation}
\label{palpha}
p_\alpha(n,k):= \frac{1-\alpha}{ n - k\alpha} = \BP( L_n = 1\giv M_k = n ) = \BP( X_n = \min_{1 \le i \le n } X_i \giv M_k = n)
\end{equation}
which is the common conditional probability in every such Gibbs$(\alpha)$ model that a new minimal value is discovered at time $n$, given that the $k$th new value is discovered at time $n$.
Then for each $k = 1,2, \ldots$
\begin{equation}
\label{nkident}
\BE( P_k ) = \BP(X = k) =  \BE \left[ p_\alpha(M_k,k) \prod_{j= k+1}^\infty \bigl( 1 - p_\alpha(M_j,j)\bigr) \right] .
\end{equation} 
\end{corollary}
\begin{proof}
The value of the common conditional probability declared above is read from Corollary \ref{crl:ocrp}, with $n$ instead of $n+1$
and $k-1$ instead of $k$, as is the fact that  given the entire sequence $M_1, M_2, \ldots$
the events of new minima at these times are conditionally independent with  probabilities $p_\alpha(M_j,j)$ for $j \ge 1$. The event $(X=k)$ is the event that the $k$th new value to occur in the sample is minimal value of the whole 
sample, so all new values discovered later are not minimal. Hence the second equality in \eqref{nkident} by conditioning on this sequence. The first equality in \eqref{nkident} 
is Corollary \ref{cor:X=X_1}.
\end{proof}

It is hard to imagine how this formula \eqref{nkident} could be checked in any other way for a general Gibbs$(\alpha)$ partition, though
Griffiths and Span\`o \cite{MR2336601} offer a deep study of the sequence $(M_1, M_2, \ldots, )$ derived from a Gibbs$(\alpha)$ partition which might provide an alternate approach.
For $\GEM(\alpha,\theta)$ there is at least a simple product formula for $\BE_{\alpha,\theta} P_k$. 
But the expected product seems very difficult to check, because there is no independence to work with.
For $\GEM(0,\theta)$ the product is quite manageable however.
%Only in the case of $\GEM(0,\theta)$ is the description of these times very simple. 
Then it is well known \cite{MR2032426} %%% AUTHOR = {Arratia, Richard and Barbour, A. D. and Tavar{\'e}, Simon},  TITLE = {Logarithmic combinatorial structures: a probabilistic
% maybe there is more appropriate reference?
that  the indicators $\Delta_n$ are independent, with $\BP_{0,\theta}(\Delta_n=1) = \theta/( \theta +n - 1)$ for $n \ge 1$.
It follows easily that for $k \ge 1$ the $\BP_{0,\theta}$ distribution of $M_k$ is given by the formula
\begin{equation}
\label{probsN}
\BP_{0,\theta} (M_k = n) = C_{n-1,k-1} \theta^k/ (\theta)_n, \qquad n\ge 1,
\end{equation}
where $C_{n,k} = (-1)^{n+k} S_{n,k}$ is the unsigned Stirling number of the first kind giving the number of permutations of $[n]$ with $k$ cycles.
We observe that the evaluation, with $(x)_r:= \Gamma(x + r)/\Gamma(x)$,
\begin{equation}
\label{momenttheta}
\BE_{0,\theta} \left[ \frac{ 1 }{ (M_k + \theta )_r  } \right] =  \frac{ \theta^{k-1}} { (\theta + r )^k (\theta + 1)_{r-1} }, 
\qquad  r > -\theta,
\end{equation}
is an elementary consequence of the fact that these probabilities in \eqref{probsN} sum to $1$ for each $\theta >0$.
This neat formula for inverse Pochhammer moments 
of $M_k$ does not seem to be well known. We only noticed it after needing the case $r=1$ to complete the check indicated below.

For $\alpha=0$ the probability $p_0(n,k)=1/n$ does not depend on $k$, and the identity \eqref{nkident} reduces easily to
\begin{equation}
\label{nkident1}
\frac{ \theta ^{k-1}}{(1 + \theta)^k } = \BE _{0,\theta} \left[ \frac{1}{M_k} \prod_{m= M_{k} + 1 }^\infty \left( 1 - \frac{\Delta_m }{m}  \right) \right]
\end{equation} 
Using independence of the $\Delta_n$ we can compute
\begin{align}
\notag
\BE _{0,\theta} \left[ \prod_{m= M_{k} + 1 }^\infty \left( 1 - \frac{\Delta_m }{m}  \right) \giv M_k = n  \right] &=  \BE _{0,\theta} \left[ \prod_{m= n+ 1 }^\infty \left( 1 - \frac{\Delta_m }{m}   \right) \right] \\
&=  \prod_{m= n+ 1 }^\infty \left( 1 - \frac{\theta }{(\theta + m - 1) m}   \right)  = \frac{n}{n + \theta}
\end{align} 
by the factorization 
$$
 1 - \frac{\theta }{(\theta + m - 1) m}   = \frac{ (m-1)}{m} \frac{ (m+ \theta)}{ m-1 + \theta}
$$
which telescopes the product.
Plugging this into \eqref{nkident1} reduces it to \eqref{momenttheta} for $r=1$.

In the same vein, conditioning on $P_1$ gives access to  $M_X= \min \{n \ge 1 : X_n = 1 \} = \max \{n: L_n = 1 \}$:
$$
\BP _{\alpha,\theta}( M_X > n ) = \BE_{\alpha,\theta} ( 1 - P_1)^n = \frac{ ( \alpha + \theta)_n }{ ( 1  + \theta )_n }
$$
which reduces to $\theta/(\theta + n)$ for $\alpha = 0$. In that case differencing gives
\begin{equation}
\label{MX=n}
\BP_{0,\theta} ( M_X =n  ) = \frac{ \theta } { ( \theta + n ) ( \theta + n -1 )}\,.
\end{equation}
On the other hand,  from the description with $L_n$, the usual $(0,\theta)$ restaurant model, and the fact mentioned above that 
if the $n$th customer goes to a new table, this table is placed to the extreme left with probability $1/n$,
\begin{align}
\notag
\BP_{0,\theta} ( M_X =n  ) &= \BP_{0,\theta}(\Delta_n = 1, L_n = 1, L_m = 0 \mbox{ for all } m > n ) \\
 &=  \frac{ \theta } { (n-1 + \theta ) } \frac{1 }{n} \prod_{ \ell = n+1}^\infty \left(  1 - \frac{\theta}{( \ell - 1 + \theta) } \frac{1 } {\ell } \right) ,
\end{align}
and this is again a telescoping product which reduces to \eqref{MX=n}.
It is not obvious how to perform the same check for general $\alpha$, because the $\Delta_n$ are no longer independent.
%%%Following Griffiths and  Span\`o \cite{MR2336601}, it may be possible to do something by conditioning on all the discovery times.

%\renewcommand{\thesection}{\Alph{section}}
%\setcounter{section}{0}
\appendix
%\addcontentsline{toc}{section}{Appendix: pseudo-size-biased orderings}
\section{Appendix: pseudo-size-biased orderings}
\label{sec:pseudo}

We need to extend a well-known notion of a size-biased permutation of a
finite or countably infinite index set $I$, or of a collection of clusters or components of some kind $C_i, i \in I$ that is indexed by $I$,
for some notion of sizes $s(C_i)$ of the clusters involved
\cite{MR1104569} %% , AUTHOR = {Donnelly, Peter}, TITLE = {The heaps process, li
\cite{MR1387889}.  %%, AUTHOR = {Pitman, Jim}, TITLE = {Random discrete distributions invariant under size-biased permutation},
Typically $s(C_i)$ is 
some kind of measure of $C_i$.
But other {\em pseudo-size} functions $s$ may also be considered, subject to the requirements that $s(C_i) >0$ for all $i$ and that $\Sigma:= \sum_{i} s(C_i)  < \infty$, 
which need only be met almost surely.  Given some collection of random components $(C_i, i \in I)$, and a pseudo-size function $s$,  an {\em $s$-biased pick}
from these components is $C_{\sigmapi(1)}$, where $\sigmapi(1) \in I$ is a random index with
\begin{equation}
\BP( \sigmapi(1) = h \giv C_i, i \in I) = s(C_h) /\Sigma \qquad (h \in I ).
\end{equation}
An {\em $s$-biased random permutation} of $(C_i, i \in I)$ is an exhaustive random indexing of components $C_{\sigmapi(j)}$
defined by a sequence of $s$-biased picks without replacement from these components, indexed by $j\in [k]$ 
if there are a finite number $k$ of components, or by $j \in \BN$ if there are an infinite number of them.
So, conditionally given $(C_i, i \in I)$,
\begin{itemize}
\item $C_{\sigmapi(1)}$ is an $s$-biased pick from $(C_i, i \in I)$,
\item given also $\sigmapi(1)$ and there is more than one component, $\sigmapi(2)$ is an $s$-biased pick from $(C_i, i \in I \setminus \{\sigmapi(1) \})$;
\item given also $\sigmapi(2)$ and there are more than two components, $\sigmapi(3)$ is an $s$-biased pick from $(C_i, i \in I \setminus \{ \sigmapi(1), \sigmapi(2) \})$, 
and so on. 
\end{itemize}

\noindent 
By a {\em size-biased permutation} one usually means the $s$-biased permutation as defined above, for the specific choice of the size function $s(C_i)$
equal to the number of elements for a finite set $C_i$, or 
some measure such as length for infinite sets $C_i$ like intervals.

Intuitively, think in terms of a bag of balls $C_i$ with pseudo-sizes $s(C_i)$ reflecting the ease with which they are drawn relative to other balls. Then an
$s$-biased random permutation of the $C_i$ is a listing of the balls in the order they are drawn in an exhaustive process of sampling without replacement.

We need this notion just for components which are blocks of a random set partition.
If the 
pseudo-size function depends just on the size of a component then the pseudo-size-biased permutation of an exchangeable random  
set partition will be an exchangeable ordered set partition.
The following Lemma presents some elementary facts about this construction for a general pseudo-size function depending just
on the real size:

\begin{lemma}
\label{lmm:sbias}
Let $s$ be a strictly positive function of positive integers $m \le n$, for some fixed $n$.
As in \eqref{psize:def} define an associated  function of compositions of $n$ by
\begin{equation}
\label{psize:def-repeat}
\psize( n_1, \ldots, n_k):= \prod_{i=1}^k \frac{ s(n_i) }  { s(n_i) + \cdots + s(n_k) }  .
\end{equation}
Suppose that $(n_1, \ldots, n_k)$ is the list of ordinary sizes of components $(C_1, \ldots, C_k)$ of a fixed ordered partition of\/ $[n]$.
Let $(C_{\sigmapi(1)}, \ldots, C_{\sigmapi(k)})$ be an $s$-biased random permutation of $(C_1, \ldots, C_k)$, for $C_i$ asigned pseudo-size $s(\#C_i) = s(n_i)$.
Then:
\begin{itemize}
\item[\upshape (i)] $(n_{\sigmapi(1)}, \ldots, n_{\sigmapi(k)})$ is an $s$-biased random permutation of  $(n_1, \ldots, n_k)$. 
\item[\upshape (ii)] For\/ $\sigmapi$ the random permutation of\/ $[k]$ so defined, \eqref{psize:def-repeat} gives the probability that\/ $\sigmapi$ is the identity, meaning the components are selected in their original order.
\item[\upshape (iii)]
For each $\sigma$ in the set $\perms{k}$ of all permutations of\/ $[k]$
\begin{equation}
\label{sigm:pi}
\BP( \sigmapi = \sigma ) = \psize( n_{\sigma(1)}, \ldots, n_{\sigma(k)} )   .
\end{equation}
\item[\upshape (iv)] For every composition $(n_1, \ldots, n_k)$ and every pseudo-size function $s$, there is the identity
\begin{equation}
\label{sigm:pisum}
\sum_{\sigma \in \perms{k} } \psize( n_{\sigma(1)}, \ldots, n_{\sigma(k)})    = 1 .
\end{equation}
\end{itemize}
\end{lemma}
\begin{proof}
Part (i) follows easily from the definition of an $s$-biased permutation, as does \eqref{sigm:pi} and its special case stated in (ii), by multiplication of successive conditional probabilities.
Finally, (iv) follows from (iii) by the law of total probability.
\end{proof}

The operation of $s$-biasing is one easy way to turn an exchangeable  
random partition of $[n]$ into an
ordered exchangeable  random partition of $[n]$. After $s$-biasing we are dealing with
an ordered exchangeable random partition of $[n]$.
The notions of the $\EPPF$ of an exchangeable random partition and the $\OEPPF$ of an ordered exchangeable random partitions were 
introduced in Section \ref{sec:partitionsfromsamples}.
The following lemma records a fundamental relation between  the $\EPPF$  of $\Pi_n$ and the
$\OEPPF$ of its $s$-biased random permutation $\tilPi_n$.

\begin{lemma}
\label{lmm:oeppf}
Let $\tilPi_n$ be the ordered exchangeable random partition of $[n]$ obtained from an exchangeable random partition $\Pi_n$ of $[n]$ by
putting its components in an $s$-biased random order. Then the\/ $\OEPPF$ $\tilp$ of\/ $\tilPi_n$ and the\/ $\EPPF$ $p$ of\/ $\Pi_n$ are related by
\begin{equation}
\label{oeppf}
\tilp (n_1, \ldots, n_k) = \psize(n_1, \ldots, n_k) p(n_1, \ldots, n_k)
\end{equation}
where $\psize$ is defined by \eqref{psize:def-repeat}.
%$\psize (n_1, \ldots, n_k):=  \prod_{i=1}^k s(n_i) (s(n_i) + \cdots + s(n_k) )^{-1}$ as in \eqref{psize:def}.
\end{lemma}
\begin{proof}
For any {\em particular} ordered  partition $(C_1, \ldots, C_k)$ of $[n]$ with components of sizes $(n_1, \ldots, n_k)$, 
the $s$-biased permutation of components of $\Pi_n$ equals $(C_1, \ldots, C_k)$ iff $\Pi_n$ equals $\{C_1, \ldots, C_k\}$,
which happens with probability $p(n_1, \ldots, n_k)$, and given that event the $s$-biased permutation puts these components in exactly the
desired order, which according to \eqref{sigm:pi} happens with probability $\psize(n_1, \ldots, n_k)$.
\end{proof}

\begin{proof}[Proof of Lemma \ref{lmm:order}]
It does not change the distribution  of $\Ns_{\bullet:n}$ to assume that the $s$-biased random ordering is made at the level of clusters say $\{C_1, \ldots, C_k\}$ of $\Pi_n$. 
Formula \eqref{s-bias} can then be understood as follows.
According to the previous lemma,  the right-hand side without the multinomial coefficient gives the probability that the $s$-biased permutation of clusters of $\Pi_n$ results in any
{\em particular} ordered partition $(C_1, \ldots, C_k)$ with clusters of these sizes. But the number of such ordered partitions of $[n]$ with the given cluster
sizes is the multinomial coefficient, and the cases are equiprobable, so the conclusion follows.

As for the converse, for a general  random composition $\Ns_{\bullet:n}$  with probability function $q(n_1, \ldots, n_k)=\BP[\Ns_{\bullet:n}=(n_1,\ldots, n_k)]$,
it is known \cite[(4)]{MR2122798}, %%%AUTHOR = {Gnedin, Alexander and Pitman, Jim}, TITLE = {Regenerative composition structures},
by arguments much as above, that the $\EPPF$ say $\pcomp(n_1, \ldots, n_k)$ of the exchangeable random partition of $[n]$ with the
same distribution of ranked component sizes as those of $\Ns_{\bullet:n}$ is determined by a summation of $q(n_{\sigma(1)}, \ldots, n_{\sigma(k)})$ over all permutations $\sigma$ of $[k]$,
weighted by the inverse of the multinomial coefficient appearing in \eqref{s-bias}. In the present context, assuming that
$q(n_1, \ldots, n_k)$ is given by the right side of \eqref{s-bias}, the multinomial coefficient cancels its inverse, and the general formula for
$\pcomp$ becomes
\begin{equation}
\label{tilpform}
\pcomp(n_1, \ldots, n_k) =  \sum_{\sigma \in \perms{k}} \psize (n_{\sigma(1)}, \ldots, n_{\sigma(k)}) \, p(n_{\sigma(1)}, \ldots, n_{\sigma(k)})  .
\end{equation}
If $p$ is symmetric, then $p(n_{\sigma(1)}, \ldots, n_{\sigma(k)}) \equiv p(n_1, \ldots, n_k)$ can be factored out of the sum,
and the remaining sum is $1$ by \eqref{sigm:pisum}. So $\pcomp = p$.
\end{proof}

\noindent
Some instances of Lemma  \ref{lmm:order} are as follows:

\begin{itemize}
\item The case $s(m):= m$ is ordinary size-biasing. Then the coefficient of $p(n_1, \ldots, n_k)$ on the right side of \eqref{s-bias} becomes
\begin{equation}
\label{count:dt}
\hspace{-1em}
{n \choose n_1, \ldots , n_k } \left( \prod_{i=1}^k  \frac{n_i } {n_i + \cdots + n_k} \right) = \frac{n!} {(n_1-1)! \cdots (n_k-1)!} \left( \prod_{i=1}^k  \frac{1} {n_i + \cdots + n_k} \right) 
\end{equation}
This instance of formula \eqref{s-bias} was given in \cite{MR1337249} and %%, AUTHOR = {Pitman, Jim}, TITLE = {Exchangeable and partially exchangeable random partitions},
\cite[(2.6)]{MR2245368}  %%% CSP 
for the ordinary size-biasing involved when $\Ns_{\bullet:n} := \Nstar_{\bullet:n}$ is the sequence of cluster sizes of $\Pi_n$ in order of appearance.
In this case, the coefficient of $p(n_1, \ldots, n_k)$ displayed in \eqref{count:dt} is a positive integer, the number of partitions of $[n]$ with
$k$  clusters of sizes $n_1, \ldots, n_k$ in order of appearance, as indicated by Donnelly and Tavar\'e in connection with their case $\alpha = 0$ of Theorem \ref{thm:sbm}.

\item If $s(m) \equiv 1$ then $\psize(n_1, \ldots, n_k) = 1/k!$. Then \eqref{s-bias} gives the probability function of the component sizes of $\Pi_n$ 
presented in a random order which given $K_n = k$ is uniform on all permutations of $[k]$. This formula appears in 
\cite[(2.7)]{MR2245368}.  %%% CSP 
It is of particular interest for sampling from $\GEM(\alpha,\alpha)$, 
when it gives the distribution of the composition of $n$ derived by uniform sampling from the interval partition generated by excursions away from $0$ of a 
standard Brownian bridge for $\alpha = 1/2$, and by a standard Bessel bridge of dimension $2-2\alpha$ for $0 < \alpha <1$. See 
\cite[\S 4.5]{MR2245368}.  %%% CSP 

\item The pseudo-size function $s(m):= m - \alpha$ is involved in Theorems \ref{thm:sbm} and \ref{thm:gibbs}.
\end{itemize}

\section{The regenerative ordering of a $\GEM(\alpha,\theta)$ sample}
\label{sec:regenerative}
This appendix compares and contrasts 
\begin{itemize}
\item the value-ordered cluster sizes $\Nright_{\bullet:n}$ in a sample $X_1, \ldots, X_n$ from a $\GEM(\alpha,\theta)$ distribution on $\{1,2, \ldots \}$,
\end{itemize}
which is the primary concern of this article, and
\begin{itemize}
\item the value-ordered cluster sizes $\NYup_{\bullet:n}$ in a sample $Y_1, \ldots, Y_n$ from a particular random discrete distribution $F_{\alpha,\theta}$ on $(0,\infty)$ constructed in 
\cite{MR2122798}, %%%%%, AUTHOR = {Gnedin, Alexander and Pitman, Jim}, TITLE = {Regenerative composition structures},
with a regenerative property,  whose atoms in size-biased order have $\GEM(\alpha,\theta)$ distribution.
\end{itemize}
See also \cite{MR2684164} for a nice review of these and related concepts.
Recall from 
\cite{MR2122798} %%, %%%%%, AUTHOR = {Gnedin, Alexander and Pitman, Jim}, TITLE = {Regenerative composition structures},
that a sequence of random compositions 
$(N_{\bullet:n}, n = 1,2, \ldots)$  is called {\em regenerative} if deletion of the first component of $N_{\bullet:n}$ of some given size $n_1 < n$ produces a copy of $N_{\bullet:n-n_1}$:
\begin{equation}
(N_{2:n},  N_{3:n}, \ldots \giv N_{1:n} = n_1 ) \ed (N_{1:n-n_1},  N_{2:n-n_1}, \ldots ) \mbox{ for each } 1 \le n_1  < n .
\end{equation}
It was shown in  
\cite{MR2120107} 
that if in sampling from a random discrete distribution on the line, the value-ordered sample frequencies from various sample sizes $n$ are regenerative in this sense,
then the distribution of these value-ordered sample frequencies is uniquely determined by that of the size-ordered frequencies (Kingman's partition structure), or equally by that of the appearance-ordered frequencies, 
which are in distribution just a size-biased rearrangement of the size-ordered sample frequencies. So a random discrete distribution $\Pbul$, or its associated partition structure, is called {\em regenerative} iff there exists 
such a regenerative rearrangement of its frequencies.
The study of such regenerative composition structures was motivated by the appearance of these structures in the interval partitions generated by the excursions of a Brownian motion or
other Markov process whose zero set is the range of a stable subordinator of index $\alpha \in (0,1)$.  

For any random interval partition of $[0,1]$, defined
by some sequence of interval components $(I_j)$, say $I_j = (G_j,D_j)$  with lengths $P_j:= D_j - G_j$, there is a canonical construction of a random discrete distribution $F$ on $[0,1]$ which 
puts mass $P_j$ at the right end of $I_j$.  The sample $Y_1, Y_2, \ldots $ from $F$ is then constructed from an i.i.d.\ uniform $[0,1]$ random sample $U_1, U_2, \ldots$ by setting $Y_i = D_j$ if $U_i \in (G_j,D_j)$.
The value-ordered clusters in the sample $Y_1, \ldots, Y_n$ then reflect the order structure  of the intervals $(I_j)$ to the extent it is revealed by the intervals discovered by the uniform sampling points $U_i$.

First we emphasize the similarity between these two models of value-ordered cluster sizes $\Nright_{\bullet:n}$ and
$\NYup_{\bullet:n}$ considered above.
The cluster sizes in order of appearance in the two sampling schemes  are identically distributed, as $\GEM(\alpha,\theta)$.
If $\alpha = 0$, the $F_{0,\theta}$ mentioned above simply puts probability $P_j$ at $1-\prod_{i=1} ^j (1 - H_i)$ where the $H_i$ are the i.i.d.\ beta$(1,\theta)$
factors driving the stick-breaking construction \eqref{stickbreak} of the $\GEM(0,\theta)$ frequencies. The order structure of these possible values is identical to that of their positive  integer labels $j = 1, 2, \ldots$.
So the value-ordered cluster sizes $\Nright_{\bullet:n}$ and $\NYup_{\bullet:n}$ are identically distributed. 

And now the big difference. For $0 < \alpha <1$, the random discrete distribution $F_{\alpha,\theta}$ involved in the regenerative ordering of $\GEM(\alpha,\theta)$ frequencies cannot have its atoms 
listed in increasing order like this. Consequently, the value-ordered sample frequencies $\Nright_{\bullet:n}$ and $\NYup_{\bullet:n}$ cannot be identically distributed for all $n$.  
There is some flexibility in the definition of $F_{\alpha, \theta}$, corresponding to change of variables from $Y_n$ to $g(Y_n)$ by a continuous and strictly increasing function $g$.
But that has no effect on the distribution of value-ordered sample frequencies $\NYup_{\bullet:n}$.
According to the results of \cite{MR2122798} %%, %%%%%, AUTHOR = {Gnedin, Alexander and Pitman, Jim}, TITLE = {Regenerative composition structures},
for $0 < \alpha < 1$, in any representation of the regenerative composition structure associated with $\GEM(\alpha,\theta)$ by value-ordered samples from  a random discrete distribution $F_{\alpha,\theta}$ on the line,
the atoms of $F_{\alpha,\theta}$ must with probability one accumulate at the left end of the support of $F_{\alpha,\theta}$, corresponding the fact that as the compositions of $n$ grow, with probability one new singleton clusters 
are added infinitely often at the extreme left end of the sample values. For large $n$ the initial components of $\NYup_{\bullet:n}$ are all small, with convergence in distribution to $(1,1, \ldots )$
as $n \to \infty$, which is not very interesting. On the other hand, the initial component $\Nright_{1:n}$ has limiting relative frequency $n^{-1} \Nright_{1:n} \to P_1 >0$ almost surely.

While the limiting behavior of these differently ordered sampling compositions derived from $\GEM(\alpha,\theta)$ is very different for $0 < \alpha < 1$, the stochastic mechanism by which they
can be described turns out to be very similar. This involves just a slight extension of the notion of pseudo-size-biased random ordering as in Lemma \ref{lmm:order}.

The definition of an $s$-biased  random permutation proposed in
Section~\ref{sec:pseudointro} and treated further in Appendix \ref{sec:pseudo} 
admits an obvious generalization in which a strictly positive function $s(m)$ of a single integer variable $m$ with $1 \le m \le n$
is replaced by  strictly positive function $s(n',n'')$ of positive integer variables $1 \le n'' \le  n' \le n$  which
for each fixed $n'$ gives the pseudo-size to be assigned to each component of size $n'' \le n'$ in making a pseudo-size-biased pick from clusters of sizes $n''_1, \ldots,  n_j''$ with $\sum_{i=1}^j n_i'' = n'$.
Then we can formulate the following straightforward extension of Lemmas \ref{lmm:sbias} and  \ref{lmm:order}.

\begin{lemma}
\label{lmm:order:gen}
Let $s = s(n',n'')$ be some arbitrary strictly positive pseudo-size function of positive integers $1 \le n'' \le n' \le n$ for some fixed positive integer $n$.
Extend the definition \eqref{psize:def} in Lemma \ref{lmm:sbias} to
\begin{equation}
\label{psize2:def}
\psize( n_1, \ldots, n_k):= \prod_{i=1}^k \frac{ s(\Nsum_i,n_i) }  { s(\Nsum_i,n_i) + \cdots + s(\Nsum_i,n_k) }   \mbox{ where } \Nsum_i:= n_i + \cdots + n_k.
\end{equation}
Then all four parts of Lemma \ref{lmm:sbias} remain valid, as does  the sampling formula of Lemma \eqref{lmm:order} for the
probability function of an $s$-biased random ordering of the cluster sizes of an exchangeable random partition $\Pi_n$ with $\EPPF$ $p$.
\end{lemma}

According to \cite[Theorem 8.1]{MR2122798} %%%, for the particular choice of pseudo-size function,
for $0 \le \alpha < 1$ and $\theta \ge 0$, in sampling from the regenerative arrangement of $\GEM(\alpha, \theta)$ frequencies,
the probability function of the value-ordered frequencies is given by a simple product formula, which in association with the simple product formula
for the $(\alpha,\theta)$ $\EPPF$, which can be read from the formulas 
\eqref{gibbseppf}--\eqref{gibbsw}--\eqref{althweights},
 is an expression of the fact
\cite[Corollary 8.2]{MR2122798} %%%, 
that these value-ordered frequencies are in an $s$-biased random order for the pseudo-size function
\begin{equation}
\label{snn}
s_{\alpha,\theta}(n',n'') = \alpha (n' - n'')  + \theta n'' 
\end{equation}
which satisfies the strict positivity requirement only for $0 \le \alpha < 1$ and $\theta \ge 0$.
For $\alpha >0$ this is quite a strange notion pseudo-size:  a linear combination of the usual size $n''$ of a component, 
and the size $n'-n''$ of its complement in a universe of size $n'$. 
For $\alpha = 0$, $s_{0,\theta}(n',n'') = \theta n''$, the constant factor $\theta$ has no effect, the pseudo-size-biasing reduces to
ordinary size-biasing, and we recover the case $\alpha = 0$ of Theorem \ref{thm:sbm}  due to Donnelly and Tavar\'e.
These results of \cite{MR2122798} %%%, for the particular choice of pseudo-size function,
can now be seen in a broader context of descriptions of  random compositions of $n$ derived from 
each other, or from random partitions of $n$, by various schemes of pseudo-size-biased sampling.
This operation of pseudo-size-biased sampling is a particularly tractable case of the more general
notion of a {\em deletion kernel} for recursive sampling of parts of a partition of $n$, as treated further in 
\cite{MR2120107}. %%% AUTHOR = {Gnedin, Alexander and Pitman, Jim}, TITLE = {Regenerative partition structures}, JOURNAL = {Electron. J. Combin.},
The present approach of working with ordered partitions of the set $[n]$, as in the proof of Lemma \ref{lmm:order}, and in some passages of
\cite{MR2122798}, %%%, for the particular choice of pseudo-size function,
seems to be technically easier than the formalism of unordered partitions of $n$ adopted in \cite{MR2120107}. %%% AUTHOR = {Gnedin, Alexander and Pitman, Jim}, TITLE = {Regenerative partition structures}, JOURNAL = {Electron. J. Combin.},

\section{Acknowledgement}

Thanks to Matthias Winkel for comments on an earlier version of this article and to the anonymous referee for helpful comments.

\bibliographystyle{amsplain}
\bibliography{gemmax,0gibbs}
\end{document}